\documentclass[oneside,english]{amsart}
\usepackage[T1]{fontenc}
\usepackage[latin9]{inputenc}
\usepackage{amstext}
\usepackage{amsthm}
\usepackage{amssymb}
\usepackage{graphicx}
\usepackage{esint}
\usepackage{hyperref}

\makeatletter
\numberwithin{equation}{section}
\numberwithin{figure}{section}
\theoremstyle{plain}
\newtheorem{thm}{\protect\theoremname}
\theoremstyle{plain}
\newtheorem{lem}[thm]{\protect\lemmaname}
\theoremstyle{plain}
\newtheorem{prop}[thm]{\protect\propositionname}
\theoremstyle{remark}
\newtheorem{rem}[thm]{\protect\remarkname}
\theoremstyle{remark}
\newtheorem{notation}[thm]{\protect\notationname}
\theoremstyle{remark}
\newtheorem{example}[thm]{\protect\examplename}

\usepackage{todonotes}

\global\long\def\Re{\operatorname{Re}}

\global\long\def\Im{\operatorname{Im}}

\global\long\def\Arg{\operatorname{Arg}}

\global\long\def\Log{\operatorname{Log}}

\makeatother

\usepackage{babel}
\providecommand{\lemmaname}{Lemma}
\providecommand{\propositionname}{Proposition}
\providecommand{\remarkname}{Remark}
\providecommand{\theoremname}{Theorem}
\providecommand{\notationname}{Notation}
\providecommand{\examplename}{Example}

\begin{document}
\title[Exponential Sheffer sequences]{The zero locus and some combinatorial properties of certain exponential Sheffer sequences}
\author{Gi-Sang Cheon${}^1$}
\address{${}^1$Department of Mathematics/ Applied Algebra and Optimization Research Center, Sungkyunkwan University, Suwon 16419, Rep. of Korea}
\email{gscheon@skku.edu}
\author{Tam\'{a}s Forg\'{a}cs${}^2$}
\address{${}^2$Department of Mathematics, California State University, Fresno, Fresno, CA 93740-8001, USA}
\email{tforgacs@csufresno.edu}
\author{Arnauld Mesinga Mwafise${}^1$}
\email{arnauld@skku.edu}
\author{Khang Tran${}^2$}
\email{khangt@csufresno.edu}
\begin{abstract} 
We present combinatorial and analytical results concerning a Sheffer sequence with an exponential generating function of the form $G(s,z)=e^{czs+\alpha z^{2}+\beta z^{4}}$, where $\alpha, \beta, c \in \mathbb{R}$ with $\beta<0$ and $c\neq 0$. We demonstrate that the zeros of all polynomials in such a Sheffer sequence are either real, or purely imaginary. Additionally, using the properties of Riordan matrices we show that our Sheffer sequence satisfies a three-term recurrence relation of order 4, and we also exhibit a connection between the coefficients of these Sheffer polynomials and the number of nodes with a a given label in certain marked generating trees.
\end{abstract}
\maketitle

\section{Introduction}
The current work can be rightfully considered as a sequel to a recent paper of the the first, second and fourth authors regarding Sheffer sequences, their zeros, and some combinatorial properties (see \cite{CFKT}). In that work we give a brief overview of Sheffer sequences and their history, as well as a short description of some of the recent works concerning the study of the zero distribution of polynomial sequences satisfying various types of recurrence relations. We refer the reader for more details on these topics to the introduction of that paper. In the paper at hand, we study Sheffer sequences with exponential generating functions. As the reader will note, every Sheffer sequence is generated by a function of the form $g(z)e^{sf(z)}$, where $g$ and $f$ are formal power series (c.f. Equation \eqref{sheffer}). We focus on a subfamily for which $f(z)=cz$ and $g(z)=\exp(\alpha z^2+\beta z^4)$. From an analytical perspective, the lack of singularities of such generating functions is appealing. In particular, the Cauchy  integral representation of the generated polynomials is versatile via deformations of the path of integration, and as such, is rather conducive to an application of the saddle point method when developing their asymptotic properties.\\
The main result of the paper concerning the zeros of certain exponential Sheffer sequences is the following theorem.
\begin{thm}
\label{thm:maintheorem} Given any $\alpha, \beta, c \in \mathbb{R}$ with $\beta<0$ and $c\neq 0$, the zeros of every polynomial in the sequence $\left\{ P_{m}(s)\right\} _{m=0}^{\infty}$
generated by 
\[
(\dag) \qquad \sum_{m=0}^{\infty}\frac{P_{m}(s)}{m!}z^{m}=e^{czs+\alpha z^{2}+\beta z^{4}} \qquad (s,z \in \mathbb{C})
\]
are either real or purely imaginary. 
\end{thm}
We point out that this result is stronger than the results found in a string of recent works on the zero distribution of various polynomial sequences (see for example \cite{CFKT},  \cite{tk1}, \cite{tk2}, \cite{tk3}) in that the present paper provides the exact curve on which the zeros of the $P_m$s lie \textit{for all m}, not just for $m \gg1$. We are able to establish the main result for all $P_m$ due to a simple differential recurrence relation the $P_m$s must satisfy (see the opening discussion of Section \ref{sec:theproof}), essentially identifying the the shift operator $P_m \stackrel{\Delta}{\longrightarrow} P_{m-1}$ as scaled differentiation -- a hyperbolicity preserving linear operator.\\
We remark that the current problem shares a heuristic trait with those studied in \cite{tk1}, \cite{tk2}, \cite{tk3}. As noted in these works, the choice for the particular families of generating functions was motivated by the central role hyperbolic polynomials (and in particular the polynomials $(1+x)^n$) play in the theory of hyperbolicity preserving linear operators on $\mathbb{R}[x]$. The fact that the function  $e^{\alpha z+\beta z^{2}}, \ \beta \leq 0$ belongs to the Laguerre-P\'olya class (see \cite{ps} for example) means that it is a locally uniform limit in $\mathbb{C}$ of polynomials $g_k(z)$ with only real zeros. This is turn implies that $g_k(z^2)$ has zeros that are either real or purely imaginary, providing a heuristic reason as to why we might expect the zeros of the generated sequence in $(\dag)$ to lie on the real and imaginary axes. \\
The methods employed in our current paper are similar to those used in the above cited works. For a discussion of how these methods compare to some others in the literature studying the asymptotic location of the zeros of polynomials we refer the reader to \cite{tk3} and the references contained therein. \\
The rest of the paper is organized as follows. In Section \ref{sec:theproof} we develop an integral representation for polynomials which are closely related to those generated by $(\dag)$. In Sections \ref{sec:globalasymp} and \ref{sec:localasymp} we provide asymptotic expressions for this integral. Finally in Section \ref{sec:changeofargh} we complete the proof of Theorem \ref{thm:maintheorem}.  The paper concludes with Section \ref{sec:combinatorics}, which contains two combinatorial results concerning the polynomial sequence under study. 

\section{The proof of Theorem \ref{thm:maintheorem}}\label{sec:theproof}

Sections \ref{sec:theproof} through \ref{sec:changeofargh} of the paper are dedicated to the proof of Theorem \ref{thm:maintheorem}. The substitutions
$cs \mapsto s$ and $\sqrt[4]{|\beta|}z \mapsto z$ transform the problem to an equivalent one of proving that for any $a\in\mathbb{R}$ the zeros of the polynomials in the sequence
$\left\{ H_{m}(s)\right\} _{m=0}^{\infty}$ generated by 
\begin{equation}\label{eq:H_mgen}
\sum_{m=0}^{\infty}\frac{H_{m}(s)}{m!}z^{m}=e^{sz+az^{2}-z^{4}}
\end{equation}
are either real or purely imaginary. Finally, substituting $z$ by $iz$ -- if necessary -- allows one to consider only $a\ge0$. Differentiating both sides of \eqref{eq:H_mgen} with respect to $s$ yields 
\[
\sum_{m=0}^{\infty}\frac{H'_{m}(s)}{m!}z^{m}=ze^{sz+az^{2}-z^{4}}=z\sum_{m=0}^{\infty}\frac{H_{m}(s)}{m!}z^{m}.
\]
It follows that 
\begin{equation}\label{eq:diffrec}
H'_{m}(s)=mH_{m-1}(s), \qquad m \geq 1.
\end{equation}
One deduces immediately that for all $m \geq 0$, the degree of $H_{m}(s)$ is $m$ and
$H_{m}(s)$ is an even (odd) polynomial if $m$ is even (odd). Relation \eqref{eq:diffrec} also implies that once $H_m$ has only real or purely imaginary zeros, so do all $H_k$ with $k <m$: 
\begin{lem}\label{lem:diffrec}
If the zeros of $H_{m}(s)$ are either real or purely imaginary then
so are the zeros of $H_{m-1}(s)$. 
\end{lem}

\begin{proof}
In the case $m$ is odd, $0$ is a zero of $H_{m}(z)$. The relations
$H'_{m}(s)=mH_{m-1}(s)$ and $H'_{m}(is)=imH_{m-1}(is)$ along with the mean
value theorem imply that the number of real or purely imaginary zeros
of $H_{m-1}(z)$ is at least $m-1$ and the lemma follows. In the
case $m$ is even, we may write $H_{m}(z)=P(z^{2})$ for some hyperbolic polynomial
$P$. The result now follows from 
\[
\left(P(z^{2})\right)'=2zP'(z^{2}),
\]
where $P'$ is again a hyperbolic polynomial. 
\end{proof}
Thus, in order to establish the main result, one needs only to prove that the zeros of $H_{m}(s)$ are either real or purely imaginary \textit{for all large} $m$. 
To study the zeros of $H_{m}(s)$ for large $m$, we start with the integral representation  
\[
H_{m}(s)=\frac{m!}{2\pi i}\ointctrclockwise_{|z|=\epsilon}\frac{e^{sz+az^{2}-z^{4}}}{z^{m+1}}dz,
\]
and employ the substitution
$z=m^{1/4}z$ and $s=m^{3/4}s$ to obtain
\begin{equation} \label{eq:saddlephi}
H_{m}(m^{3/4}s)=\frac{m!}{m^{m/4}2\pi i}\ointctrclockwise_{|z|=\epsilon}e^{m\phi(z,s)}e^{\sqrt{m}az^{2}}\frac{dz}{z},
\end{equation}
where 
\begin{align}
\phi(z,s) & =sz-z^{4}-\Log z\label{eq:phidef},
\end{align}
and $\Log z$ denotes the principal branch of the logarithm. We shall use the saddle point method (see for example \cite[Ch.4]{temme}) to give asymptotic estimates for the right hand side of \eqref{eq:saddlephi}. We thus need to find the critical points of the exponent of integrand in \eqref{eq:saddlephi}, i.e. we need to solve \begin{equation}
\phi_{z}(z,s)+\frac{2az}{\sqrt{m}}=0.\label{eq:zeta_def}
\end{equation}
The reader will note that one solution to \eqref{eq:zeta_def} is given by
\begin{equation}
\zeta=\frac{1}{2}\sqrt{\frac{\eta}{6}+\frac{A}{12\sqrt[3]{2}}+\frac{\eta^{2}+48}{6\ 2^{2/3}A}}-\frac{1}{2}\sqrt{\frac{\eta}{3}-\frac{A}{12\sqrt[3]{2}}+\frac{s}{2\sqrt{\frac{\eta}{6}+\frac{A}{12\sqrt[3]{2}}+\frac{\eta^{2}+48}{6\ 2^{2/3}A}}}-\frac{\eta^{2}+48}{6\ 2^{2/3}A}},\label{eq:zetadef}
\end{equation}
where 
\begin{eqnarray*}
A&=&\sqrt[3]{-2\eta^{3}+288\eta-\sqrt{\left(2\eta^{3}-288\eta-108s^{2}\right)^{2}-4\left(\eta^{2}+48\right)^{3}}+108s^{2}}, \textrm{and} \\
\eta&=&\frac{2a}{\sqrt{m}}.
\end{eqnarray*}
Let
\begin{eqnarray}
J_{1} &=&\left(0,2^{5/2}/3^{3/4}-\frac{2a}{\sqrt[4]{12}\sqrt{m}}\right), \label{eq:J1}\\
J_{2} &=&\left(0,2^{5/2}/3^{3/4}+\frac{2a}{\sqrt[4]{12}\sqrt{m}}\right) \label{eq:J2},
\end{eqnarray}
and
\[
J:=J_{1}\cup iJ_{2}.
\]
Then the map $\zeta(s)$ (as defined by \eqref{eq:zetadef}) maps $J\cup\{0\}$ to the curve given in Figure \ref{fig:zeta}. 
\begin{center}
\begin{figure}[h]
\begin{centering}
\includegraphics[width=\textwidth]{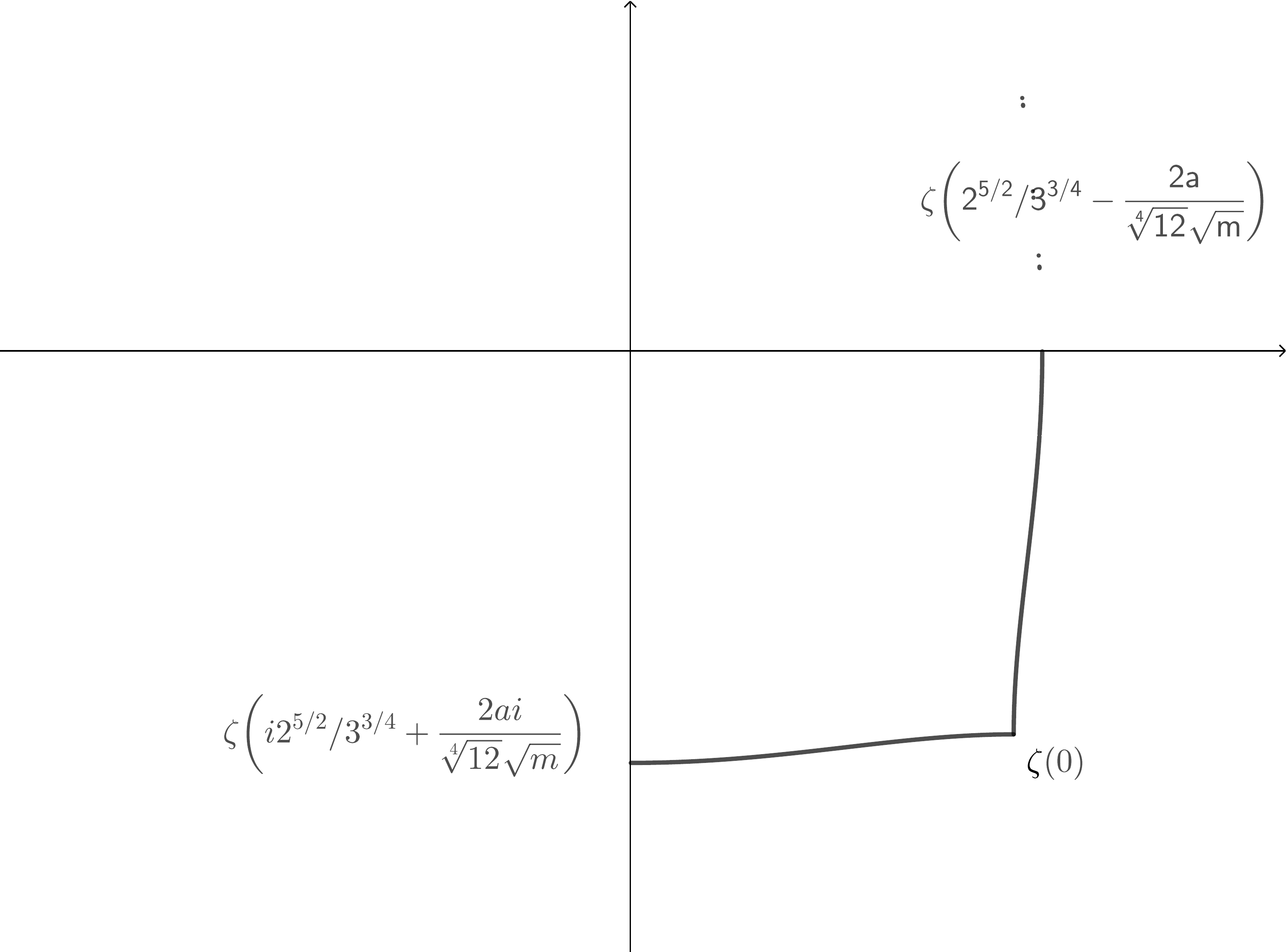} 
\par\end{centering}
\caption{\label{fig:zeta}The curve $\zeta(s)$}
\end{figure}
\par\end{center}
For the purposes of our analysis, a key property of the $\zeta(s)$ curve is that it lies entirely in the fourth quadrant, which we now demonstrate.
\begin{lem}
\label{lem:zetaloc}For any $s\in J$, $\zeta(s)$ lies in the open fourth
quadrant, and it is the only solution of \eqref{eq:zeta_def} there.
\end{lem}

\begin{proof}
We set 
\begin{eqnarray*}
f(z) & = & sz-4z^{4}-1+\frac{2az^{2}}{\sqrt{m}}\\
h(z) & = & sz,\quad\textrm{and}\\
g(z) & = & -4z^{4}-1+\frac{2az^{2}}{\sqrt{m}}=-4\left(\left(z^{2}-\frac{a}{4\sqrt{m}}\right)^{2}+\left(\frac{1}{4}-\frac{a^{2}}{16m}\right)\right).
\end{eqnarray*}
With this notation, the zeros of $f$ -- other than the one at the origin -- are precisely those complex numbers that satisfy \eqref{eq:zeta_def}. In addition, it is clear that for $m \gg 1$, $g$ has one zero in each quadrant. Thus, by showing that $f$ and $g$ have the same number of zeros in the fourth quadrant, we demonstrate that \eqref{eq:zeta_def} has a unique solution in the fourth quadrant for all values of $s \in J$.\\
For $\varepsilon\ll1$, consider the polar rectangle $te^{i\theta},\varepsilon<t<\sqrt[4]{\pi/3},\ -\pi/2<\theta<0$.
We divide the boundary of $R$ into four pieces: the two arcs, and
the two line segments. On the arc $|z|=\sqrt[4]{\pi/3}$, we have
\begin{eqnarray*}
|g(z)|>\pi>3\sqrt[4]{\pi/3}>|h(z)|\qquad\left(\frac{a}{\sqrt{m}}\ll1\right).
\end{eqnarray*}
On the arc $|z|=\varepsilon$, $-\pi/2\leq\theta\leq0$ we compute
\begin{eqnarray*}
|g(z)| & > & 1-4\varepsilon^{4}-\frac{2a\varepsilon^{2}}{\sqrt{m}}\\
 & > & 3\varepsilon>|h(z)|\qquad\left(\frac{a}{\sqrt{m}}\ll1\right).
\end{eqnarray*}
We now turn our attention to the real line segment $\varepsilon\leq t\leq\sqrt[4]{\frac{\pi}{3}}$, $\theta=0$.
For $s\in J$, we compute 
\begin{eqnarray*}
|g(z)|-|h(z)| & \geq & 4t^{4}+1-|s|t-\frac{2at^{2}}{\sqrt{m}}=1+4t^{4}-t\frac{2^{5/2}}{3^{3/4}}-\frac{2at}{\sqrt{m}}\left(t+\frac{1}{\sqrt[4]{12}}\right)\\
 & > & \frac{1}{2}+\mathcal{O}\left(\frac{1}{\sqrt{m}}\right).
\end{eqnarray*}

Finally, we look at the line segment $z=it,\varepsilon\leq t\leq\sqrt[4]{\frac{\pi}{3}}$.
Since $g(t)=g(it)$ for all $t\in\mathbb{R}$, and $|h(it)|\leq|si|t+\frac{2at^{2}}{\sqrt{m}}$
still hold on this segment, the same analysis we just gave in the
previous paragraph also shows that $|g(z)|>|h(z)|$ on this segment
as well. Rouch\'e's theorem now implies that $g$, and $f=g+h$ have the
same number of zeros on this polar rectangle. Since $g$ has exactly
one zero here, so does $f$, for all $s\in J$. Since \eqref{eq:zetadef})
is a solution of $f(z)=0$, we make the choices of the square roots
in the definition so that $\zeta(s)$ belongs to the fourth quadrant
for all $s\in J$.
\end{proof}
\begin{rem}
\label{rem:critsfixedzeta} Using arguments analogous to those in Lemma \ref{lem:zetaloc},
one can also show that for each $s\in J$, $\zeta$ is the only solution
(in $z$) to the equation 
\[
\phi_{z}\left(z,s+\frac{2a\zeta}{\sqrt{m}}\right)=0
\]
in the open fourth quadrant. 
\end{rem}

The following asymptotic expressions for $\zeta(s)$ can be easily verified by a computer algebra system, and will be needed for the local asymptotic analysis we carry out in Section \ref{sec:localasymp}.
\begin{lem}
\label{lem:zetalarges}For $s\in J_{1}$ and $2^{5/2}/3^{3/4}-s=o(1)$,
\begin{align*}
\zeta(s)\sim & \frac{1}{\sqrt[4]{12}}-\frac{i\sqrt{\sqrt[4]{12}(2^{5/2}/3^{3/4}-s)-\eta}}{2\sqrt{6}}+\frac{4\eta-\sqrt[4]{12}(2^{5/2}/3^{3/4}-s)}{24\sqrt{2}3^{3/4}}.
\end{align*}
For $s\in J_{2}$ and $i2^{5/2}/3^{3/4}-s=o(1)$,
\[
\zeta(s)\sim-\frac{i}{\sqrt[4]{12}}+\frac{\sqrt{\sqrt[4]{12}(2^{5/2}/3^{3/4}+is)+\eta}}{2\sqrt{6}}+\frac{4i\eta-\sqrt[4]{12}(i2^{5/2}/3^{3/4}-s)}{24\sqrt{2}3^{3/4}}.
\]
\end{lem}
The next proposition provides a key element in our proof of the main result, by establishing the existence of a curve, along which we find a suitable asymptotic expansion of the integral representing the polynomials in question. Proposition \ref{lem:zyfunc} is the analog of Proposition 24 in \cite{CFKT}, but is appreciably simpler to establish. The nature (and singularities) of the generating function in \cite{CFKT} necessitated a lengthy discussion on the topology of the level sets of $\Re \phi(z(y),s)$ when extending the curve from a local piece around $\zeta(s)$ to (complex) infinity. No such discussion is necessary in the current work, since the generating function of the Sheffer sequence under consideration has no singularities.  For notational simplicity, from here on out we suppress the dependence of $\zeta(s)$ on $s$, and will simply write $\zeta$.

\begin{prop}
\label{lem:zyfunc}For each $s\in J$, there exist $-\infty\le K<0<L\le\infty$,
and a function $z(y)$ analytic in a neighborhood of $(K,L)$ such
that 
\begin{itemize}
\item[(i)] $z(0)=\zeta$, 
\item[(ii)] $z(K)\in i\mathbb{R}^{-}$, $z(L)\in\mathbb{R}^{+}$, 
\item[(iii)] $z(y)$ lies in the open fourth quadrant for all $y\in(K,L)$, \quad \textrm{and}
\item[(iv)]
\begin{equation}
-y^{2}=\phi\left(z(y),s+\frac{2a\zeta}{\sqrt{m}}\right)-\phi\left(\zeta,s+\frac{2a\zeta}{\sqrt{m}}\right),\qquad \textrm{for all} \  y\in (K,L).\label{eq:zyfunc}
\end{equation} 
\end{itemize}
In the case $K=-\infty$ and/or $L=\infty$, condition (ii) is replaced
by $\lim_{y\rightarrow-\infty}\Re z(y)=0$ and/or $\lim_{y\rightarrow\infty}\Im z(y)=0$
respectively.
\end{prop}

\begin{proof}
By the definition of $\zeta$, we have 
\[
\phi\left(z,s+\frac{2a\zeta}{\sqrt{m}}\right)-\phi\left(\zeta,s+\frac{2a\zeta}{\sqrt{m}}\right)=\frac{\phi_{z^{2}}(\zeta,s+2a\zeta/\sqrt{m})}{2}(z-\zeta)^{2}(1+b(z)),
\]
where $b(z)$ is analytic in an open neighborhood of $\zeta$ and
$|b(z)|<1$. These properties of $b(z)$ imply that we may invert the relation 
\[
y=\frac{\sqrt{\phi_{z^{2}}(\zeta,s+2a\zeta/\sqrt{m})}}{\sqrt{2}i}(z-\zeta)\sqrt{1+b(z)}
\]
to obtain $z(y)$, analytic in a small open neighborhood of $0$, such
that $z(0)=\zeta$ and that \eqref{eq:zyfunc} holds. To extend this open
neighborhood, let $L$ be the supremum of the set of $l\in\mathbb{R}^{+}$
such that $z(y)$ has an analytic continuation to an open neighborhood
of $[0,l)$ and $z([0,l))$ is a subset of the fourth quadrant. Let
$y_{k}\in[0,L)$ such that $y_{k}\rightarrow L$ and the sequence
$z(y_{k})$ is convergent to an extended complex number denoted by $z(L)$.
It is clear from \eqref{eq:zyfunc} that $z(L)\ne\zeta$ and $z(L)\ne0$.
We note that $\phi(z,s+2a\zeta/\sqrt{m}$) is analytic in an open
neighborhood of $z(L)$ and, by Remark \ref{rem:critsfixedzeta},
\[
\phi_{z}\left(z(L),s+\frac{2a\zeta}{\sqrt{m}}\right)\ne0
\]
since $z(L)\ne\zeta$ by \eqref{eq:zyfunc}. The implicit function
theorem combined with the definition of $L$ now implies that either $z(L)=\infty$, $z(L)\in i\mathbb{R}^{-}$,
or $z(L)\in \mathbb{R}^{+}$. We claim that if $z(L)=\infty$, then the curve $z(y)$ approaches either the positive real axis, or the negative imaginary axis. Indeed, writing $z=re^{i\theta}$ and considering the imaginary parts of the two sides of \eqref{eq:zyfunc}, we see that for large $r$ 
\[
r^{4}\sin4\theta=\mathcal{O}(r).
\]
This in turn implies that as $r\rightarrow\infty$, either $\theta\rightarrow0$ ($z$ approaches the positive real axis) or $\theta\rightarrow-\pi/2$ ($z$ approaches the negative imaginary axis).

Similarly, if we let $K$ be the infimum of the set of $k\in\mathbb{R}^{-}$
such that $z(y)$ has an analytic continuation to an open neighborhood
of $(k,0]$ and $z((k,0])$ is a subset of the fourth quadrant, then either
$z(L)\in i\mathbb{R}^{-}$, $z(L)\in\mathbb{R}^{+}$ or $z(K)=\infty$
(in which case the curve approaches the positive real, or the negative imaginary axis).

 Note that $z(L)$ and $z(K)$ cannot both belong to $\mathbb{R}^{+}$. If they did, the fact that
\[
\Im\phi\left(x,\frac{2a\zeta}{\sqrt{m}}\right)=\frac{2ax}{\sqrt{m}}\Im\zeta
\]
is monotone in $x\in\mathbb{R}^{+}$ would imply that the imaginary part of right hand side of \eqref{eq:zyfunc} takes on different values at $z(K)$ and $z(L)$, unless $z(L)=z(K)$. In this latter case, however, we would conclude that $\Im\phi(z,2a\zeta/\sqrt{m})$
is constant on the closed curve $z(y)$, $L\le y\le K$, which would imply that $\phi(z,2a\zeta/\sqrt{m})$ -- by virtue of  being analytic on $\mathbb{C} \setminus (-\infty,0]$ -- is constant on the domain bounded by $z(y)$, a contradiction.

The same argument shows that $z(L)$ and $z(K)$ cannot both belong to $i\mathbb{R}^{-}$. In the case $z(L)=z(K)=\infty$, we cannot
have $\lim_{y\rightarrow\infty}\Re z(y)=\lim_{y\rightarrow-\infty}\Re z(y)=0$
or $\lim_{y\rightarrow\infty}\Im z(y)=\lim_{y\rightarrow-\infty}\Im z(y)=0$,
since as $y\rightarrow\infty$, \eqref{eq:zyfunc} has four distinct
solutions in $z$ approaching $\infty$ along the positive/negative
real/imaginary axes. The result now follows, with swapping $z(y)$ and $z(-y)$ if necessary.
\end{proof}
\begin{notation}\label{not:Gamma}
Using the notation of Proposition \ref{lem:zyfunc} we write $\Gamma$ for the curve $z(y)$, $K<y<L$ oriented counterclockwise.
\end{notation}
We note that $\bar \Gamma \subset$ QI, $-\Gamma \subset$ QII, and $-\bar \Gamma \subset$ QIII. In addition, the curve $\Gamma \cup \bar \Gamma \cup -\Gamma \cup -\bar \Gamma$ is either closed, or can be `truncated' by two or four small line segments to form a closed curve containing the origin. \\
Our next result provides integral representations for the $H_m$s, which we will later use to estimate the number of zeros of the each polynomial from below.
\begin{lem}
\label{lem:Im-rep} Let $J_1$ and $J_2$ be as defined by \eqref{eq:J1} and \eqref{eq:J2}. If $s\in J_{1}$, then
\begin{equation} \label{eq:HmJ1}
H_{m}(m^{3/4}s)=\frac{m!}{m^{m/4}\pi}\Im\left(\int_{\Gamma}\frac{e^{msz+a\sqrt{m}z^{2}-mz^{4}}}{z^{m+1}}dz-\int_{-\overline{\Gamma}}\frac{e^{msz+a\sqrt{m}z^{2}-mz^{4}}}{z^{m+1}}dz\right).
\end{equation}
On the other hand, if $s\in iJ_{2}$, then for even $m$ we have
\begin{equation} \label{eq:HmJ2e}
H_{m}(m^{3/4}s)=\frac{m!}{m^{m/4}\pi}\Im\left(\int_{\Gamma}\frac{e^{msz+a\sqrt{m}z^{2}-mz^{4}}}{z^{m+1}}dz-\int_{\overline{\Gamma}}\frac{e^{msz+a\sqrt{m}z^{2}-mz^{4}}}{z^{m+1}}dz\right),
\end{equation}
while for odd $m$,
\begin{equation} \label{eq:HmJ2o}
H_{m}(m^{3/4}s)=\frac{m!}{m^{m/4}\pi i}\Re\left(\int_{\Gamma}\frac{e^{msz+a\sqrt{m}z^{2}-mz^{4}}}{z^{m+1}}dz-\int_{\overline{\Gamma}}\frac{e^{msz+a\sqrt{m}z^{2}-mz^{4}}}{z^{m+1}}dz\right).
\end{equation}
\end{lem}

\begin{proof}
Note that for any large $R$ and any small $\epsilon$, we have
\[
\left|\int_{\pm R-i\epsilon}^{\pm R+i\epsilon}\frac{e^{msz+a\sqrt{m}z^{2}-mz^{4}}}{z^{m+1}}dz\right|\le\int_{-\epsilon}^{\epsilon}\frac{e^{-mR^{4}/2}}{(R+\epsilon)^{m+1}}dt\rightarrow0
\]
as $R\rightarrow\infty$ and $\epsilon\text{\ensuremath{\rightarrow0}}$. Similarly,
\[
\left|\int_{\pm iR-\epsilon}^{\pm iR+\epsilon}\frac{e^{sz+az^{2}-z^{4}}}{z^{m+1}}dz\right|\rightarrow0
\]
as $R\rightarrow\infty$, $\epsilon\rightarrow0$. 
The Cauchy integral formula and the remark preceding the statement of this Lemma together imply that
\begin{align}
\frac{2\pi im^{m/4}}{m!}H_{m}(m^{3/4}s) & =\int_{\Gamma}\frac{e^{msz+a\sqrt{m}z^{2}-mz^{4}}}{z^{m+1}}dz-\int_{\overline{\Gamma}}\frac{e^{msz+a\sqrt{m}z^{2}-mz^{4}}}{z^{m+1}}dz\nonumber \\
 & +\int_{-\Gamma}\frac{e^{msz+a\sqrt{m}z^{2}-mz^{4}}}{z^{m+1}}dz-\int_{-\overline{\Gamma}}\frac{e^{msz+a\sqrt{m}z^{2}-mz^{4}}}{z^{m+1}}dz.\label{eq:fourcontours}
\end{align}

If $s\in J_{1}$, then in particular $s \in \mathbb{R}$, and hence  
\begin{eqnarray*}
\int_{\pm\Gamma}\frac{e^{msz+a\sqrt{m}z^{2}-mz^{4}}}{z^{m+1}}dz&=&\int_{\pm\overline{\Gamma}}\frac{e^{ms\overline{z}+a\sqrt{m}\overline{z}^{2}-m\overline{z}^{4}}}{\overline{z}^{m+1}}d\overline{z}\\
&=&\overline{\int_{\pm\overline{\Gamma}}\frac{e^{msz+a\sqrt{m}z^{2}-mz^{4}}}{z^{m+1}}dz},
\end{eqnarray*}
 from which  \eqref{eq:HmJ1} follows. Representations \eqref{eq:HmJ2e} and \eqref{eq:HmJ2o} can be obtained similarly, using the identities \begin{align*}
\int_{\pm\Gamma}\frac{e^{msz+a\sqrt{m}z^{2}-mz^{4}}}{z^{m+1}}dz & =\int_{\mp\overline{\Gamma}}\frac{e^{ms(-\overline{z})+a\sqrt{m}(-\overline{z})^{2}-m(-\overline{z})^{4}}}{(-\overline{z})^{m+1}}d(-\overline{z})\\
 & =(-1)^{m}\overline{\int_{\mp\overline{\Gamma}}\frac{e^{msz+a\sqrt{m}z^{2}-mz^{4}}}{z^{m+1}}dz}.
\end{align*}
\end{proof}
To demonstrate that the zeros of $H_{m}(m^{3/4}s)$ are either real or purely
imaginary, we count the number of zeros of this polynomial on $J_{1}\cup iJ_{2}$
and compare the count with the degree of $H_m$. Using the representations we have just obtained, along with the argument principle, we get a lower bound on the number of zeros of $H_m$ by computing the change in the argument of
\[
h_{m}(s):=\left\{ \begin{array}{cc} \displaystyle{ \int_{\Gamma}\frac{e^{msz+a\sqrt{m}z^{2}-mz^{4}}}{z^{m+1}}dz-\int_{-\overline{\Gamma}}\frac{e^{msz+a\sqrt{m}z^{2}-mz^{4}}}{z^{m+1}}dz} & \textrm{if} \ s \in J_1 \\ 
\displaystyle{\int_{\Gamma}\frac{e^{msz+a\sqrt{m}z^{2}-mz^{4}}}{z^{m+1}}dz-\int_{\overline{\Gamma}}\frac{e^{msz+a\sqrt{m}z^{2}-mz^{4}}}{z^{m+1}}dz} & \textrm{if} \ s \in iJ_2 \end{array}
\right. .
\]
 We find the change in the argument of $h_m$ (in Section \ref{sec:changeofargh}) using an asymptotic representation of it, which we now develop.


\section{\label{sec:globalasymp}A global asymptotic formula for $h_m(s)$}

In this section, we will apply the saddle point method to find an
asymptotic formula (which is uniform in $s$) for $h_{m}(s)$ as $m\rightarrow\infty$
when $s$ is away from the two endpoints of $J_{1}$ and $iJ_{2}$.
In particular, for this whole section, we assume that 
\[
(\dag) \quad \left\{ \begin{array}{cc} s-2^{5/2}/3^{3/4}=\Omega(1) & \textrm{if} \ s\in J_{1}, \\ s-i2^{5/2}/3^{3/4}=\Omega(1) & \textrm{if} \ s \in iJ_{2},\\
|s| \gg\ln m/m & \textrm{if} \ s\in J. \end{array} \right.
\]
 We will show that the first integral 
\[
\int_{\Gamma}\frac{e^{msz+a\sqrt{m}z^{2}-mz^{4}}}{z^{m+1}}dz
\]
in the definition of $h_{m}(s)$ dominates the second integral, and that
the main term of the approximation (c.f. \eqref{eq:mainterm}) comes
from the integral along a small portion of $\Gamma$ around the saddle
point $\zeta$. As a consequence 
\[
h(s)\sim e^{m\phi(\zeta,s+2a\zeta/\sqrt{m})-\sqrt{m}a\zeta^{2}}\frac{\sqrt{2\pi}ie^{-i\Arg(\phi_{z^{2}}(\zeta,s+2a\zeta/\sqrt{m}))/2}}{\zeta\sqrt{m}\sqrt{\left|\phi_{z^{2}}(\zeta,s+2a\zeta/\sqrt{m})\right|}}.
\]
For the remainder of the section, we focus on the case $s\in J_{1}$.
The case $s\in iJ_{2}$ follows from similar arguments.

\subsection{The main term of the approximation}

Differentiating (\ref{eq:phidef}) repeatedly we find  
\begin{align}
\phi_{z^{2}}\left(\zeta,s+\frac{2a\zeta}{\sqrt{m}}\right) & =-12\zeta^{2}+\frac{1}{\zeta^{2}}\stackrel{(\dag)}{=}\Omega(1), \label{eq:Phiz^2Omega(1)}\\
\phi_{z^{3}}\left(\zeta,s+\frac{2a\zeta}{\sqrt{m}}\right) & =-24\zeta-\frac{2}{\zeta^{3}}, \nonumber\\
\phi_{z^{4}}\left(\zeta,s+\frac{2a\zeta}{\sqrt{m}}\right) & =-24+\frac{6}{\zeta^{4}}, \quad \textrm{and} \nonumber \\
\phi_{z^{k}}\left(\zeta,s+\frac{2a\zeta}{\sqrt{m}}\right)& =\frac{(-1)^{k}(k-1)!}{\zeta^{k}}, \qquad k\ge5. \nonumber
\end{align}
 We expand $\phi(z,s+2a\zeta/\sqrt{m})$ (as a function of $z$) in a Taylor series  centered at $\zeta$ and rearrange to obtain
\begin{eqnarray}
\phi(z,s+2a\zeta/\sqrt{m})-\phi(\zeta,s+2a\zeta/\sqrt{m})&=&\frac{\phi_{z^{2}}(\zeta,s+2a\zeta/\sqrt{m})}{2!}(z-\zeta)^{2} \nonumber\\
&+& \sum_{k=3}^{\infty} \frac{\phi_{z^{k}}(\zeta,s+2a\zeta/\sqrt{m})}{k!}(z-\zeta)^k. \label{eq:taylorexp}
\end{eqnarray}
Equation \eqref{eq:Phiz^2Omega(1)} implies that 
\begin{equation}
\frac{1}{\phi_{z^{2}}(\zeta,s+2a\zeta/\sqrt{m})}=\mathcal{O}(1). \label{eq:recipphi_z^2}
\end{equation}
Therefore, for $|z-\zeta| \ll 1$, equation \eqref{eq:zyfunc} and the above Taylor expansion give  the following estimate:
\[
-y^{2}=\frac{\phi_{z^{2}}(\zeta,s+2a\zeta/\sqrt{m})}{2}(z(y)-\zeta)^{2}\left(1+\mathcal{O}\left(z(y)-\zeta\right)\right),
\]
which we solve for $z(y)$:  
\[
z(y)=\zeta\pm\frac{\sqrt{2}ie^{-i\Arg(\phi_{z^{2}}(\zeta,s+2a\zeta/\sqrt{m}))/2}}{\sqrt{\left|\phi_{z^{2}}(\zeta,s+2a\zeta/\sqrt{m})\right|}}y+\mathcal{O}\left((z-\zeta)^{2}\right), \qquad (|z-\zeta| \ll 1).
\]
In order to obtain a similar expression for $z'(y)$, we differentiate both sides of \eqref{eq:taylorexp} and get
\[
-2y=\phi_{z^{2}}(\zeta,s+2a\zeta/\sqrt{m})(z-\zeta)z'(y)+ \sum_{k=3}^{\infty} \frac{\phi_{z^{k}}(\zeta,s+2a\zeta/\sqrt{m})}{(k-1)!}(z-\zeta)^{k-1}.
\]
Using the relation \eqref{eq:recipphi_z^2} once more, we conclude that for $|z-\zeta| \ll 1$,
\begin{eqnarray*}
-2y&=&\phi_{z^{2}}(\zeta,s+2a\zeta/\sqrt{m})(z-\zeta)z'(y)(1+\mathcal{O}(z-\zeta)) \\
&=&\phi_{z^{2}}(\zeta,s+2a\zeta/\sqrt{m})\left(\pm\frac{\sqrt{2}ie^{-i\Arg(\phi_{z^{2}}(\zeta,s+2a\zeta/\sqrt{m}))/2}}{\sqrt{\left|\phi_{z^{2}}(\zeta,s+2a\zeta/\sqrt{m})\right|}}y+ \mathcal{O}\left((z-\zeta)^{2}\right)\right)z'(y)(1+\mathcal{O}(z-\zeta))\\
&=&\phi_{z^{2}}(\zeta,s+2a\zeta/\sqrt{m})\left(\pm\frac{\sqrt{2}ie^{-i\Arg(\phi_{z^{2}}(\zeta,s+2a\zeta/\sqrt{m}))/2}}{\sqrt{\left|\phi_{z^{2}}(\zeta,s+2a\zeta/\sqrt{m})\right|}}y\right)z'(y)(1+\mathcal{O}(z-\zeta)).
\end{eqnarray*}
It now follows that
\[
z'(y)=\pm\frac{\sqrt{2}ie^{-i\Arg(\phi_{z^{2}}(\zeta,s+2a\zeta/\sqrt{m}))/2}}{\sqrt{\left|\phi_{z^{2}}(\zeta,s+2a\zeta/\sqrt{m})\right|}}\left(1+\mathcal{O}\left(z-\zeta\right)\right), \qquad (|z-\zeta| \ll 1).
\]
Using the algebraic identity
\[
\sqrt{m}az^{2}=\sqrt{m}a\left(\zeta^{2}+2\zeta(z-\zeta)+(z-\zeta)^{2}\right),
\]
we conclude for any $z=z(y)$, 
\begin{align}
 m\phi(z,s)+\sqrt{m}az^{2} &=m(sz-z^{4}-\Log z)+\sqrt{m}az^{2}\nonumber \\ & =m\phi\left(z,s+\frac{2a\zeta}{\sqrt{m}}\right)-\sqrt{m}a\zeta^{2}+\sqrt{m}a(z-\zeta)^{2}\label{eq:shifts}\\
 & =m\left(\phi\left(z,s+\frac{2a \zeta}{\sqrt{m}}\right)- \phi\left(\zeta,s+\frac{2a \zeta}{\sqrt{m}}\right)\right) \nonumber \\
 & +m\phi\left(\zeta,s+\frac{2a \zeta}{\sqrt{m}}\right)-\sqrt{m} a \zeta^2+\sqrt{m}a (z-\zeta)^2. \nonumber
\end{align}
Let $\epsilon=\frac{1}{m^{3/8}}$. Along the
curve $\Gamma_{\epsilon}:z=z(y)$, $-\epsilon<y<\epsilon$, $\sqrt{m}a (z-\zeta)^2=\mathcal{O}(1/\sqrt[4]{m})$, and hence $e^{\sqrt{m}a(z-\zeta)^2}=1+o(1)$. Thus, the integral of $e^{m\phi(z,s)}e^{\sqrt{m}az^{2}}/z$ along $\Gamma_{\epsilon}$ is asymptotically equivalent to 
\begin{eqnarray*}
&& \pm e^{m\phi(\zeta,s+2a\zeta/\sqrt{m})-\sqrt{m}a\zeta^{2}}\frac{\sqrt{2}ie^{-i\Arg(\phi_{z^{2}}(\zeta,s+2a\zeta/\sqrt{m}))/2}}{\zeta\sqrt{\left|\phi_{z^{2}}(\zeta,s+2a\zeta/\sqrt{m})\right|}}\int_{-\epsilon}^{\epsilon}e^{-my^{2}}dy \\
&\stackrel{\sqrt{m}y\rightarrow y}{=}&
\pm e^{m\phi(\zeta,s+2a\zeta/\sqrt{m})-\sqrt{m}a\zeta^{2}}\frac{\sqrt{2}ie^{-i\Arg(\phi_{z^{2}}(\zeta,s+2a\zeta/\sqrt{m}))/2}}{\zeta\sqrt{m}\sqrt{\left|\phi_{z^{2}}(\zeta,s+2a\zeta/\sqrt{m})\right|}}\int_{-\sqrt{m}\epsilon}^{\sqrt{m}\epsilon}e^{-y^{2}}dy.
\end{eqnarray*}
Since $\sqrt{m}\epsilon\rightarrow\infty$, this expression is asymptotic
to 
\begin{equation}
\pm e^{m\phi(\zeta,s+2a\zeta/\sqrt{m})-\sqrt{m}a\zeta^{2}}\frac{\sqrt{2\pi}ie^{-i\Arg(\phi_{z^{2}}(\zeta,s+2a\zeta/\sqrt{m}))/2}}{\zeta\sqrt{m}\sqrt{\left|\phi_{z^{2}}(\zeta,s+2a\zeta/\sqrt{m})\right|}}.\label{eq:mainterm}
\end{equation}

\subsection{The error term} \label{sec:errorterm}
In order to conclude that $h_m(s) \sim$ \eqref{eq:mainterm}, we need to demonstrate that the integral of $e^{m\phi(z(y),s)}e^{\sqrt{m}az(y)^{2}}$ over the tails of $\Gamma$ is small (in terms of \eqref{eq:mainterm}), as is its integral over $\bar \Gamma$.  We now establish the necessary asymptotic bounds for 
\begin{equation}
\int_{\epsilon}^{L}e^{m\phi(z(y),s)}e^{\sqrt{m}az(y)^{2}}\frac{z'(y)}{z(y)}dy\label{eq:tailint},
\end{equation}
-- which will also bound the integral over $(K, -\epsilon)$,--  as well as for the integral
\[
\int_{-\overline{\Gamma}}e^{m\phi(z,s)}e^{\sqrt{m}az^{2}}\frac{dz}{z}.
\]
Using equation \eqref{eq:shifts}, we rewrite the expression in \eqref{eq:tailint} as 
\[
e^{m\phi\left(\zeta,s+\frac{2a\zeta}{\sqrt{m}}\right)-\sqrt{m}a\zeta^{2}}\int_{\epsilon}^{L}e^{-my^{2}+\sqrt{m}a(z-\zeta)^{2}}\frac{z'(y)}{z(y)}dy.
\]
The relation 
\begin{eqnarray*}
-y^{2}&=&\phi\left(z(y),s+\frac{2a\zeta}{\sqrt{m}}\right)-\phi\left(\zeta,s+\frac{2a\zeta}{\sqrt{m}}\right)\\
&=&(z-\zeta)\left(s+\frac{2a\zeta}{\sqrt{m}} \right)-z^4-\Log z+\zeta^4+\Log \zeta
\end{eqnarray*}
implies that
\begin{eqnarray*}
1/z(y)&=&\mathcal{O}(1), \qquad \textrm{and}\\
z(y)&=&\mathcal{O}(\sqrt{y}).
\end{eqnarray*}
Similarly,
\[
-2y=(s+2a\zeta/\sqrt{m}-4z^{3}-1/z)z'(y),
\]
which implies that if $y$ is large, $z'(y)=\mathcal{O}(1)$. Hence, (recall that $\epsilon=1/m^{3/8}$) 
\begin{align*}
 & e^{m\phi\left(\zeta,s+\frac{2a\zeta}{\sqrt{m}}\right)-\sqrt{m}a\zeta^{2}}\int_{\epsilon}^{L}e^{-my^{2}+\sqrt{m}a(z-\zeta)^{2}}\frac{z'(y)}{z(y)}dy\\
= & \mathcal{O}\left(e^{m\phi\left(\zeta,s+\frac{2a\zeta}{\sqrt{m}}\right)-\sqrt{m}a\zeta^{2}}\int_{\epsilon}^{\infty}e^{-my^{2}+\sqrt{m}y}dy\right)\\
= & \mathcal{O}\left(e^{m\phi\left(\zeta,s+\frac{2a\zeta}{\sqrt{m}}\right)-\sqrt{m}a\zeta^{2}}\int_{\epsilon}^{\infty}e^{-my^{2}/2}dy\right)\\
= & \mathcal{O}\left(e^{m\phi\left(\zeta,s+\frac{2a\zeta}{\sqrt{m}}\right)-\sqrt{m}a\zeta^{2}}\frac{e^{-m\epsilon^{2}/2}}{m\epsilon}\right).
\end{align*}
Identical arguments show that
\begin{eqnarray*}
\int_{K}^{-\epsilon} e^{m\phi(z(y),s)}e^{\sqrt{m} a z(y)^2} \frac{z'(y)}{z(y)}dy&=&e^{m\phi\left(\zeta,s+\frac{2a\zeta}{\sqrt{m}}\right)-\sqrt{m}a\zeta^{2}}\int_{K}^{-\epsilon}e^{-my^{2}+\sqrt{m}a(z-\zeta)^{2}}\frac{z'(y)}{z(y)}dy\\
&=&\mathcal{O}\left(e^{m\phi\left(\zeta,s+\frac{2a\zeta}{\sqrt{m}}\right)-\sqrt{m}a\zeta^{2}}\frac{e^{-m\epsilon^{2}/2}}{m\epsilon}\right).
\end{eqnarray*}
Next, we find a bound for the integral
\[
\int_{-\overline{\Gamma}}e^{m\phi(z,s)}e^{\sqrt{m}az^{2}}\frac{dz}{z}.
\]
With the substitution $z$ by $-\overline{z}$, the integral becomes
\begin{align*}
 & -\int_{\Gamma}e^{m\phi(-\overline{z},s)}e^{\sqrt{m}a\overline{z}^{2}}\psi(-\overline{z})d\overline{z}\\
= & \int_{K}^{L}e^{m\phi(-\overline{z},s)}e^{\sqrt{m}a\overline{z}^{2}}\frac{\overline{z'(y)}}{\overline{z}}dy\\
= & \mathcal{O}\left(\int_{K}^{L}e^{\Re\left(m\phi(-\overline{z},s)+\sqrt{m}a\overline{z}^{2}\right)}dy\right).
\end{align*}
The reader will note that for all $y \in (K,L)$,
\begin{align*}
 & \Re\left(m\phi(-\overline{z},s)+\sqrt{m}a\overline{z}^{2}\right)\\
= & m\left(-s\Re z-\Re z^{4}-\ln|z|\right)+\sqrt{m}a\Re z^{2}\\
= & m\Re\phi(z,s)+\sqrt{m}a\Re z^{2}-2ms\Re z\\
= & \Re\left(-my^{2}+m\phi\left(\zeta,s+\frac{2a\zeta}{\sqrt{m}}\right)-\sqrt{m}a\zeta^{2}+\sqrt{m}a(z-\zeta)^{2}\right)-2ms\Re z.
\end{align*}
In the range $-\epsilon<y<\epsilon$, $\Re z>0$, $\Re z=\Omega(1)$, and 
$z-\zeta=\mathcal{O}(1/m^{3/8})$. Consequently, with $s\gg\ln m/m$, we conclude
\[
\int_{-\epsilon}^{\epsilon}e^{\Re\left(m\phi(-\overline{z},s)+\sqrt{m}a\overline{z}^{2}\right)}dy
\]
is little-oh of \eqref{eq:mainterm}. For $|y|\ge\epsilon$, we use the estimate $\sqrt{m}a(z-\zeta)^{2}=\mathcal{O}(\sqrt{m}y)=o(my^{2})$
to obtain the same conclusion for the integrals
\[
\int_{\epsilon}^{L}e^{\Re\left(m\phi(-\overline{z},s)+\sqrt{m}a\overline{z}^{2}\right)}dy\qquad\text{and}\qquad\int_{K}^{-\epsilon}e^{\Re\left(m\phi(-\overline{z},s)+\sqrt{m}a\overline{z}^{2}\right)}dy.
\]

We conclude this section with the remark that if $s\in iJ_{2}$,
one can obtain similar estimates by replacing the curve $-\overline{\Gamma}$
by $\overline{\Gamma}$ and noting that 
\begin{align*}
 & \Re\left(m\phi(\overline{z},s)+\sqrt{m}a\overline{z}^{2}\right)\\
= & m\left(|s|\Im z-\Re z^{4}-\ln|z|\right)+\sqrt{m}a\Re z^{2}\\
= & m\Re\phi(z,s)+\sqrt{m}a\Re z^{2}+2m|s|\Im z,
\end{align*}
where $\Im z\le0$. In summary, we have shown thus far that if $\ln m/m \ll s$ and $2^{5/2}/3^{3/4}-s=\Omega(1)$, then $h_m(s) \sim \eqref{eq:mainterm}$. The next section addresses the asymptotics for the remaining relevant ranges of $s$.


\section{\label{sec:localasymp}A local asymptotic formula for $h_m(s)$ }

In this section, we consider ranges of $s$ as it approaches the non-zero endpoints of $J_1$ and $iJ_2$, while establishing the dominance of the central piece of the integral on $\Gamma$ (namely on the range $\alpha<y<\beta$) both over the tails ($y<\alpha$ and $\beta<y$), as well as over the companion curve $-\bar \Gamma$. In particular, we treat the cases
\[
e^{-m^{1/8}}\ll2^{5/2}/3^{3/4}-\frac{2a}{\sqrt[4]{12}\sqrt{m}}-s=o(1)
\]
for $s\in J_{1}$, and 
\[
e^{-m^{1/8}}\ll2^{5/2}/3^{3/4}+is+\frac{2a}{\sqrt[4]{12}\sqrt{m}}=o(1)
\]
for $s\in iJ_{2}$. Let $\alpha<0<\beta$ be such that $|z(\alpha)-\zeta|=\frac{C}{m^{1/4}\ln m}=|z(\beta)-\zeta|$
for some fixed constant $C$. Let $\Gamma$ be as described in Notation \ref{not:Gamma}. We will show that
\begin{align*}
\int_{\Gamma}e^{m\phi(z,s)}e^{\sqrt{m}az^{2}}\psi(z)dz & \sim\frac{e^{-\sqrt{m}a\zeta^{2}+\phi\left(\zeta,s+\frac{2a\zeta}{\sqrt{m}}\right)}}{\zeta}\int_{\alpha}^{\beta}e^{-my^{2}}z'(y)dy.
\end{align*}
We treat the case $s\in J_{1}$ and remark that the case $s\in iJ_{2}$ follows
from similar arguments.\\
Recall that the curve $z(y)$, $K\le y\le L$, in Proposition \ref{lem:zyfunc}
satisfies $\gamma(0)=\zeta$ and

\begin{equation}\label{eq:y^2KL}
-y^{2}=\phi\left(z(y),s+\frac{2a\zeta}{\sqrt{m}}\right)-\phi\left(\zeta,s+\frac{2a\zeta}{\sqrt{m}}\right),\qquad\forall y\in (K,L).
\end{equation}
We calculate 
\begin{eqnarray*}
\phi_{z^{3}}\left(\frac{1}{\sqrt[4]{12}},2^{5/2}/3^{3/4}\right)&=&-2^{7/2}3^{3/4}, \quad \textrm{and} \\
\phi_{z^2}\left(\frac{1}{\sqrt[4]{12}},\frac{2^{5/2}}{3^{3/4}} \right)&=&0
\end{eqnarray*}
to deduce that 
\begin{equation}\label{eq:phiz^2asymp}
\phi_{z^{2}}(\zeta,s+2a\zeta/\sqrt{m})=-2^{7/2}3^{3/4}(\zeta-1/\sqrt[4]{12})+\mathcal{O}\left((\zeta-1/\sqrt[4]{12})^{2}\right).
\end{equation}
Expanding the right hand side of equation \eqref{eq:y^2KL} about $z=\zeta$ for $|z-\zeta| \ll 1$ gives 
\begin{equation}
-y^{2}=\frac{\phi_{z^{2}}(\zeta,s+2a\zeta/\sqrt{m})}{2!}(z-\zeta)^{2}+\frac{\phi_{z^{3}}(\zeta,s+2a\zeta/\sqrt{m})}{3!}(z-\zeta)^{3}+\mathcal{O}\left((z-\zeta)^{4}\right),\label{eq:gammacurveapprox}
\end{equation}
and hence using equation \eqref{eq:phiz^2asymp} we obtain
\begin{align}
-y^{2}= & -2^{5/2}3^{3/4}(\zeta-1/\sqrt[4]{12})(z-\zeta)^{2}-2^{5/2}3^{-1/4}(z-\zeta)^{3}\label{eq:zcurvelarges}\\
 & +\mathcal{O}\left((z-\zeta)^{4}+(\zeta-1/\sqrt[4]{12})^{2}(z-\zeta)^{2}+(z-\zeta)^{3}(\zeta-1/\sqrt[4]{12})\right).\nonumber \\
\sim & 2^{5/2}3^{-1/4}(z-\zeta)^{2}\left(3(\zeta-1/\sqrt[4]{12})+(z-\zeta)\right)\nonumber \\
= & 2^{5/2}3^{-1/4}i(iz-i\zeta)^{2}\left(i3(\zeta-1/\sqrt[4]{12})+(iz-i\zeta)\right), \quad y \in (\alpha, \beta).\nonumber 
\end{align}
Differentiating both sides of \eqref{eq:y^2KL} with respect to $y$ and applying similar estimates leads to the asymptotic identity
\begin{equation}
z'(y)\sim-\frac{2y}{2^{5/2}3^{3/4}i(z-\zeta)\left(2i(\zeta-1/\sqrt[4]{12})+i(z-\zeta)\right)}.\label{eq:zdiffsimlocal}
\end{equation}

\begin{rem}
\label{rem:angle} We note that there exists a fixed $\delta>0$ independent of $\zeta$, such that $\left|\pi-\Arg\left(i3(\zeta-1/\sqrt[4]{12})+(iz-i\zeta)\right)\right|>\delta$.
Indeed, if $\Arg\left(i3(\zeta-1/\sqrt[4]{12})+(iz-i\zeta)\right)\rightarrow\pi$,
then the fact that $\Arg \left(i3(\zeta-1/\sqrt[4]{12})\right)\rightarrow0$ (as $m \to \infty$)
implies $\Arg(iz-i\zeta)\rightarrow \pi$, contradicting \eqref{eq:zcurvelarges}
above. Without loss of generality, we will also write $|\Arg(iz-i\zeta)|>\delta$ using the same $\delta >0$.
\end{rem}
Using the facts established in Remark \ref{rem:angle} above, the next result gives a bound on the asymptotic convergence rate of $\alpha^2 \leadsto 0$ and $\beta^2 \leadsto 0$, as $z(\alpha)$ and $z(\beta)$ approach $\zeta$ at the rate prescribed in the definition of $\alpha$ and $\beta$. 
\begin{lem}
\label{lem:alphabetaaprox}Let $\alpha$ and $\beta$ be as defined at the beginning of the section. Then 
\[
\alpha^{2}\gg\frac{1}{m^{3/4}\ln^{3}m},\qquad\text{and}\qquad\beta^{2}\gg\frac{1}{m^{3/4}\ln^{3}m}.
\]
\end{lem}

\begin{proof}
From Lemma  \ref{lem:zetalarges} we deduce that
\[
\zeta-1/\sqrt[4]{12}\sim-\frac{i\sqrt{\sqrt[4]{12}(2^{5/2}/3^{3/4}-s)-\eta}}{2\sqrt{6}}\in i\mathbb{R^{-}}.
\]
Consequently, 
\begin{align*}
 & \left|3(\zeta-1/\sqrt[4]{12})+(z(\alpha)-\zeta)\right|^{2}\\
\sim & 9|\zeta-1/\sqrt[4]{12}|^{2}+|z(\alpha)-\zeta|^{2}+6|\zeta-1/\sqrt[4]{12}|\Im(z(\alpha)-\zeta).
\end{align*}
By Remark \ref{rem:angle}, the last expression is greater than or equal to
\[
9|\zeta-1/\sqrt[4]{12}|^{2}+|z(\alpha)-\zeta|^{2}-c|\zeta-1/\sqrt[4]{12}||z(\alpha)-\zeta|
\]
for some fixed $0<c<6$. As a quadratic polynomial in $|\zeta-1/\sqrt[4]{12}|$,
the above expression is at least as big as
\[
\left(1-\frac{c^{2}}{36}\right)|z(\alpha)-\zeta|^{2}\asymp1/m^{1/2}\ln^{2}m.
\]
We conclude from \eqref{eq:zcurvelarges} that 
\begin{align*}
\alpha^{2} & \asymp|z(\alpha)-\zeta|^{2}\left|3(\zeta-1/\sqrt[4]{12})+(z(\alpha)-\zeta)\right|\gg\frac{1}{m^{3/4}\ln^{3}m}.
\end{align*}
The computations are analogous for the estimate on $\beta^2$.
\end{proof}
Continuing the development of the asymptotics for the central integral we note that if $y\in(\alpha,\beta)$, then 
\[
\sqrt{m}az^{2}=\sqrt{m}a(\zeta^{2}+2\zeta(z-\zeta))+o(1).
\]
From the definition
\[
\phi(z,s)=sz-z^{4}-\Log z
\]
we immediately obtain that
\[
m\phi(z,s)+\sqrt{m}az^{2}=m\phi\left(z,s+\frac{2a\zeta}{\sqrt{m}}\right)-\sqrt{m}a\zeta^{2}+o(1).
\]
The following identity is now readily obtained: 
\[
\int_{\substack{\gamma(y)\\
\alpha\le y\le\beta
}
}e^{m\phi(z,s)}e^{\sqrt{m}az^{2}}\frac{dz}{z}=\frac{e^{-\sqrt{m}a\zeta^{2}+\phi\left(\zeta,s+\frac{2a\zeta}{\sqrt{m}}\right)}}{\zeta}\int_{\alpha}^{\beta}e^{-my^{2}}z'(y)dy(1+o(1)).
\]
We now take a closer look at $z'(y)$, appearing in the right hand integral above. Using the principal cut for the square root, the asymptotic identity 
\begin{equation}
y\sim\pm2^{5/4}3^{-1/8}e^{i\pi/4}(iz-i\zeta)\sqrt{3i(\zeta-1/\sqrt[4]{12})+(iz-i\zeta)} \label{eq:ysim}
\end{equation}
follows from equation \eqref{eq:zcurvelarges}. We combine \eqref{eq:ysim} and \eqref{eq:zdiffsimlocal} to conclude
that 
\begin{align}
z'(y)\sim\mp\frac{e^{i\pi/4}}{2^{1/4}3^{7/8}} & \left(\frac{2i(\zeta-1/\sqrt[4]{12})+i(z-\zeta)}{\sqrt{3i(\zeta-1/\sqrt[4]{12})+i(z-\zeta)}}\right)^{-1}\nonumber \\
=\mp\frac{e^{i\pi/4}}{2^{1/4}3^{7/8}} & \left(\sqrt{3i(\zeta-1/\sqrt[4]{12})+i(z-\zeta)}-\frac{i(\zeta-1/\sqrt[4]{12})}{\sqrt{3i(\zeta-1/\sqrt[4]{12})+i(z-\zeta)}}\right)^{-1}\nonumber \\
\sim\mp\frac{e^{i\pi/4}}{2^{1/4}3^{7/8}} & \left(\sqrt{3i(\zeta-1/\sqrt[4]{12})+i(z-\zeta)}+\frac{\sqrt{\sqrt[4]{12}(2^{5/2}/3^{3/4}-s)-\eta}}{2\sqrt{6}\sqrt{3i(\zeta-1/\sqrt[4]{12})+i(z-\zeta)}}\right)^{-1}\nonumber \\
\sim\mp\frac{e^{i\pi/4}}{2^{1/4}3^{7/8}} & \left(\sqrt{3i(\zeta-1/\sqrt[4]{12})+i(z-\zeta)}+\frac{\sqrt{\sqrt[4]{12}(2^{5/2}/3^{3/4}-s)-\eta}}{2\sqrt{6}\sqrt{3i(\zeta-1/\sqrt[4]{12})+i(z-\zeta)}}\right)^{-1}.\label{eq:zderivasymp}
\end{align}

\begin{rem}
\label{rem:anglezderiv} Remark \ref{rem:angle} and equation \eqref{eq:zderivasymp} imply that there exists small $\xi>0$ (fixed and independent of $m$) so that for
any $y\in(\alpha,\beta)$, 
\[
\Re\left(e^{3i\pi/4}z'(y)\right)\ge\xi
\]
if \eqref{eq:zderivasymp} holds with a negative sign. If the sign is positive, the reverse inequality holds.
For the sake of exposition, we assume that we are in the first case (i.e. that of the negative sign). The other case follows along the same line of argument (after replacing $\alpha$ by $\beta$ in Lemmas
\ref{lem:anglez-xi} and \ref{lem:cutxaxis}, as well as in the ensuing computations).
\end{rem}
The next two auxiliary lemmas are both used in the final asymptotic bound for the central integral. The reader will note that the orientation of the curve $z(y)$ is ambiguous as is the sign of $z'(y)$. Once we pick the sign however, the orientation becomes unambiguous, as we remark after the proof of Lemma \ref{lem:cutxaxis}.  
\begin{lem}
\label{lem:anglez-xi} If $y\in(\alpha,0)$\footnote{If equation \eqref{eq:zderivasymp} holds with the positive sign, then the statment is for $y \in (0,\beta)$}, then $-\pi/4\le\Arg(z-\zeta)\le3\pi/4$. 
\end{lem}

\begin{proof}
Remark \ref{rem:anglezderiv} implies that 
\[
\Re\left(e^{3i\pi/4}z'(y)\right)\ge0.
\]
Combining this inequality with 
\[
\Re\left(e^{3i\pi/4}(z(0)-\zeta)\right)=0
\]
gives 
\begin{align*}
\Re\left(e^{3i\pi/4}(z-\zeta)\right) & \le0,\qquad\forall y\in(\alpha,0)
\end{align*}
and the result follows.
\end{proof}
\begin{lem}
\label{lem:cutxaxis} If $y\in(\alpha,0)$, then $z(y)\notin\mathbb{R}^{+}$. 
\end{lem}

\begin{proof}
Suppose, by way of contradiction, that $z_{0}=z(y_{0})\in\mathbb{R}^{+}$ for
some $y_{0}\in(\alpha,0$). We recall from the definition of $\phi$
that 
\[
\Im\phi\left(\zeta,s+\frac{2a\zeta}{\sqrt{m}}\right)=\Im\left(s\zeta-\zeta^{4}-\Log\zeta+\frac{2a\zeta^{2}}{\sqrt{m}}\right).
\]
Note that since $s \in \mathbb{R}$, 
\begin{align*}
\frac{d}{ds}\Im\left(s\zeta-\zeta^{4}-\Log\zeta+\frac{a\zeta^{2}}{\sqrt{m}}\right) & =\Im\left(\left(s-4\zeta^{3}-\frac{1}{\zeta}+\frac{2a\zeta}{\sqrt{m}}\right)\cdot\frac{d\zeta}{ds}+\zeta\right)\\
 & =\text{\ensuremath{\Im\zeta<0.}}
\end{align*}
Thus, since 
\[
\Im\left(s\zeta-\zeta^{4}-\Log\zeta+\frac{a\zeta^{2}}{\sqrt{m}}\right) \Bigg |_{s=2^{5/2}/3^{3/4}-\frac{2a}{\sqrt[4]{12} \sqrt{m}}}=0,
\]
we see that
\begin{align*}
\Im\phi\left(\zeta,s+\frac{2a\zeta}{\sqrt{m}}\right) & =\Im\left(s\zeta-\zeta^{4}-\Log\zeta+\frac{a\zeta^{2}}{\sqrt{m}}\right)+\Im\left(\frac{a\zeta^{2}}{\sqrt{m}}\right)\\
 & >\Im\left(\frac{a\zeta^{2}}{\sqrt{m}}\right),
\end{align*}
With this inequality, the equation 
\[
\Im\phi\left(\zeta,s+\frac{2a\zeta}{\sqrt{m}}\right)=\Im\left(z_{0},s+\frac{2a\zeta}{\sqrt{m}}\right)=\frac{2az_{0}}{\sqrt{m}}\Im\zeta
\]
and the inequality $\Im(\zeta)<0$ imply that 
\[
z_{0}<\Re\zeta.
\]
We combine this inequality with Lemma \ref{lem:anglez-xi} to conclude
$\pi/2<\Arg(z_{0}-\zeta)<3\pi/4$. Consequently Lemma \ref{lem:zetalarges}
yields $|z_{0}-\zeta|\le\sqrt{2}|\zeta-1/\sqrt[4]{12}|$, and 
\[
-\arctan\frac{2}{\sqrt{3}}\le\Arg\left(3i(\zeta-1/\sqrt[4]{12})+i(z-\zeta)\right)\le\arctan\frac{2}{\sqrt{3}}.
\]
Since $-3\pi/4\le\Arg e^{i\pi/4}(iz_{0}-i\zeta)\le-\pi/2$, 
\[
\frac{1}{2}\arctan\frac{2}{\sqrt{3}}-\frac{\pi}{2}\ge\Arg e^{i\pi/4}(iz_{0}-i\zeta)\sqrt{3i(\zeta-1/\sqrt[4]{12})+(iz-i\zeta)}\ge-\frac{3\pi}{4}-\frac{1}{2}\arctan\frac{2}{\sqrt{3}},
\]
a contradiction to \eqref{eq:ysim}, which stipulates that
\[
y_{0}\sim2^{5/4}3^{-1/8}e^{i\pi/4}(iz_{0}-i\zeta)\sqrt{3i(\zeta-1/\sqrt[4]{12})+(iz-i\zeta)}.
\]
\end{proof}
With $K$ as defined in Lemma \ref{lem:cutxaxis}, the previous result implies that $(\alpha,0)\subset(K,0)$.
Using the asymptotic expansion developed for $z'(y)$ in equation \eqref{eq:zderivasymp} we obtain
\begin{align*}
 & \left|\int_{\alpha}^{\beta}e^{-my^{2}}z'(y)dy\right|\\
\asymp & \left|\int_{\alpha}^{\beta}e^{-my^{2}}\left(\sqrt{3i(\zeta-1/\sqrt[4]{12})+i(z-\zeta)}+\frac{\sqrt{\sqrt[4]{12}(2^{5/2}/3^{3/4}-s)-\eta}}{2\sqrt{6}\sqrt{3i(\zeta-1/\sqrt[4]{12})+i(z-\zeta)}}\right)^{-1}dy\right|\\
\gg & \int_{\alpha}^{\beta}e^{-my^{2}}\Re\left(\sqrt{3i(\zeta-1/\sqrt[4]{12})+i(z-\zeta)}+\frac{\sqrt{\sqrt[4]{12}(2^{5/2}/3^{3/4}-s)-\eta}}{2\sqrt{6}\sqrt{3i(\zeta-1/\sqrt[4]{12})+i(z-\zeta)}}\right)^{-1}dy\\
\ge & \int_{\alpha}^{0}e^{-my^{2}}\Re\left(\sqrt{3i(\zeta-1/\sqrt[4]{12})+i(z-\zeta)}+\frac{\sqrt{\sqrt[4]{12}(2^{5/2}/3^{3/4}-s)-\eta}}{2\sqrt{6}\sqrt{3i(\zeta-1/\sqrt[4]{12})+i(z-\zeta)}}\right)^{-1}dy.
\end{align*}
If we let 
\[
A:=\sqrt{3i(\zeta-1/\sqrt[4]{12})+i(z-\zeta)}+\frac{\sqrt{\sqrt[4]{12}(2^{5/2}/3^{3/4}-s)-\eta}}{2\sqrt{6}\sqrt{3i(\zeta-1/\sqrt[4]{12})+i(z-\zeta)}},
\]
then 
\[
\Re A^{-1}=\frac{\Re A}{|A|^{2}}\gg\frac{1}{|A|},
\]
where the last inequality comes from the fact that $|\pi-\Arg\left(3i(\zeta-1/\sqrt[4]{12})+i(z-\zeta)\right)|>\delta$.
We note that 
\begin{align*}
|A| & \le\sqrt{|3i(\zeta-1/\sqrt[4]{12})+i(z-\zeta)|}+\frac{\sqrt{\sqrt[4]{12}(2^{5/2}/3^{3/4}-s)-\eta}}{2\sqrt{6}\sqrt{|3i(\zeta-1/\sqrt[4]{12})+i(z-\zeta)|}}\\
\ll & \frac{1}{\sqrt{|3i(\zeta-1/\sqrt[4]{12})+i(z-\zeta)|}},
\end{align*}
and 
\[
|3i(\zeta-1/\sqrt[4]{12})+i(z-\zeta)|^{2}\sim9|\zeta-1/\sqrt[4]{12}|^{2}+|z-\zeta|^{2}+6|\zeta-1/\sqrt[4]{12}|\Re\left(i(z-\zeta)\right).
\]
By Lemma \ref{lem:anglez-xi}, the right side above is at least 
\[
9|\zeta-1/\sqrt[4]{12}|^{2}+|z-\zeta|^{2}-3\sqrt{2}|\zeta-1/\sqrt[4]{12}||z-\zeta|.
\]
As a quadratic polynomial in $|z-\zeta|$, this expression is at least
\[
\frac{9}{2}|\zeta-1/\sqrt[4]{12}|^{2}.
\]
Thus 
\begin{eqnarray*}
\left|\int_{\alpha}^{\beta}e^{-my^{2}}z'(y)dy\right|&\gg& \sqrt{|\zeta-1/\sqrt[4]{12}|}\int_{\alpha}^{0}e^{-my^{2}}dy\\
&=&\frac{\sqrt{|\zeta-1/\sqrt[4]{12}|}}{\sqrt{m}}\int_{\sqrt{m}\alpha}^{0}e^{-y^{2}}dy\gg\frac{\sqrt{|\zeta-1/\sqrt[4]{12}|}}{\sqrt{m}},
\end{eqnarray*}
since $\sqrt{m}\alpha\rightarrow\infty$ by Lemma \ref{lem:alphabetaaprox}.
We conclude that 
\begin{equation}
\int_{\substack{\gamma(y)\\
\alpha\le y\le\beta
}
}e^{m\phi(z,s)}e^{\sqrt{m}az^{2}}\psi(z)dz\gg e^{-\sqrt{m}a\zeta^{2}+\phi\left(\zeta,s+\frac{2a\zeta}{\sqrt{m}}\right)}\frac{\sqrt{|\zeta-1/\sqrt[4]{12}|}}{\sqrt{m}}.\label{eq:lowerboundmaintermlarges}
\end{equation}
We now consider the integral over the tail of $\gamma(y)$, for $y>\beta$. Using \eqref{eq:shifts} we find that 
\begin{equation}
\int_{\substack{\gamma(y)\\
\beta<y<L
}
}e^{m\phi(z,s)+\sqrt{m}z^{2}}\frac{dz}{z}=e^{-\sqrt{m}a\zeta^{2}+m\phi\left(\zeta,s+\frac{2a\zeta}{\sqrt{m}}\right)}\int_{\beta}^{L}e^{-my^{2}+\sqrt{m}a(z-\zeta)^{2}}\frac{z'(y)}{z(y)}dy.\label{eq:tailintlarges}
\end{equation}
The fact that $z=\mathcal{O}(\sqrt{y})$ for large $y$ (see the beginning of Section \ref{sec:errorterm}), coupled with the equivalence
\[
(z-\zeta)^{2}\asymp\frac{1}{m^{1/2}\ln^{2}m} \qquad \textrm{(for small $y$)}
\]
 imply that
\[
(z-\zeta)^{2}=\mathcal{O}\left(\frac{y}{m^{1/2}\ln^{2}m}\right).
\]
Just as we did in the central integral estimate, here we also need bounds on $z'(y)$. To this end recall that
\[
-2y=(s-4z^{3}-1/z)z'(y),
\]
which implies that for large $y$, $z'(y)=o(1)$. Using the asymptotic expression for $z'(y)$ given in equation \eqref{eq:zderivasymp} and the Cauchy inequality, we conclude that for small $y$ 
\[
z'(y)=\mathcal{O}\left(\frac{1}{|\zeta-1/\sqrt[4]{12}|}\right).
\]
Combining these results gives that for all $y \in (\beta, L)$,
\[
z'(y)=\mathcal{O}\left(\frac{1}{|\zeta-1/\sqrt[4]{12}|}\right).
\]
Thus,  
\begin{align*}
& e^{-\sqrt{m}a\zeta^{2}+m\phi\left(\zeta,s+\frac{2a\zeta}{\sqrt{m}}\right)}\int_{\beta}^{L}e^{-my^{2}+\sqrt{m}a(z-\zeta)^{2}}\frac{z'(y)}{z(y)}dy \\
=& \mathcal{O}\left(\frac{e^{-\sqrt{m}a\zeta^{2}+m\phi\left(\zeta,s+\frac{2a\zeta}{\sqrt{m}}\right)}}{|\zeta-1/\sqrt[4]{12}|}\int_{\beta}^{\infty}e^{-my^{2}+ay/\ln^{2}m}dy\right)\\
= & \mathcal{O}\left(\frac{e^{-\sqrt{m}a\zeta^{2}+m\phi\left(\zeta,s+\frac{2a\zeta}{\sqrt{m}}\right)}}{|\zeta-1/\sqrt[4]{12}|}\int_{\sqrt{m}\beta}^{\infty}e^{-y^{2}/2}dy\right)\\
= & \mathcal{O}\left(\frac{e^{-\sqrt{m}a\zeta^{2}+m\phi\left(\zeta,s+\frac{2a\zeta}{\sqrt{m}}\right)}}{|\zeta-1/\sqrt[4]{12}|}\frac{e^{-m\beta^{2}/2}}{m\beta^{2}}\right)\\
& = o \left(e^{-\sqrt{m}a\zeta^{2}+\phi\left(\zeta,s+\frac{2a\zeta}{\sqrt{m}}\right)}\frac{\sqrt{|\zeta-1/\sqrt[4]{12}|}}{\sqrt{m}} \right).
\end{align*}
The same conclusion holds -- mutatis mutandis -- if we replace $\beta$ by $\alpha$.

We now find an upper-bound for the integral 
\[
\int_{-\overline{\Gamma(y)}}e^{m\phi(z,s)}e^{\sqrt{m}az^{2}}\frac{dz}{z},
\]
which, after substituting $z$ by $-\overline{z}$, becomes
\[
\int_{-\infty}^{\infty}e^{m\phi(-\overline{z},s)}e^{\sqrt{m}a\overline{z}^{2}}\frac{\overline{z'(y)}}{\overline{z}}dy.
\]
Recall that 
\begin{align*}
 & \Re\left(m\phi(-\overline{z},s)+\sqrt{m}a\overline{z}^{2}\right)\\
= & \Re\left(-my^{2}+m\phi\left(\zeta,s+\frac{2a\zeta}{\sqrt{m}}\right)-\sqrt{m}a\zeta^{2}+\sqrt{m}a(z-\zeta)^{2}\right)-2ms\Re z.
\end{align*}
Now, if $\alpha<y<\beta$, then the definitions of $\alpha$ and $\beta$
imply that $\Re z>0$ and $\Re z=\Omega(1)$. Consequently, 
\[
\Re\left(m\phi(-\overline{z},s)+\sqrt{m}a\overline{z}^{2}\right)=o\left(-my^{2}+m\phi\left(\zeta,s+\frac{2a\zeta}{\sqrt{m}}\right)-\sqrt{m}a\zeta^{2}\right).
\]
The bound $z'(y)=\mathcal{O}(1/|\zeta-1/\sqrt[4]{12}|)$ thus implies that
\[
\int_{\alpha}^{\beta}e^{\Re\left(m\phi(-\overline{z},s)+\sqrt{m}a\overline{z}^{2}\right)}\frac{\overline{z'(y)}}{\overline{z}}dy=o\left( \frac{e^{-\sqrt{m}a\zeta^{2}+m\phi\left(\zeta,s+\frac{2a\zeta}{\sqrt{m}}\right)}}{|\zeta-1/\sqrt[4]{12}|}\frac{e^{-m\beta^{2}/2}}{m\beta^{2}}\right).
\]
If, on the other hand, $y>\beta$ or $y<\alpha$, we apply $\sqrt{m}a(z-\zeta)^{2}=\mathcal{O}(\sqrt{m}y)=o(my^{2})$
to obtain the same asymptotic bound for the integrals
\[
\int_{\beta}^{L}e^{\Re\left(m\phi(-\overline{z},s)+\sqrt{m}a\overline{z}^{2}\right)}dy\qquad\text{and}\qquad\int_{K}^{\alpha}e^{\Re\left(m\phi(-\overline{z},s)+\sqrt{m}a\overline{z}^{2}\right)}dy.
\]
Having developed the asymptotic representation for $h_m(s)$, we are now ready to compute the change of argument in $h_m(s)$, which in turn will allow us to count the number of zeros of the generated polynomials under investigation. 


\section{The change of argument of $h_m(s)$ }\label{sec:changeofargh}

Motivated by the asymptotic analysis in Sections \ref{sec:globalasymp} and \ref{sec:localasymp}, we define\footnote{This function $g$ clearly depends on $m$, as does $h_m(s)$. For the ease of exposition we drop this dependence from the notation}
\begin{align}
g(\zeta(s)) & =e^{m\phi(\zeta,s+2a\zeta/\sqrt{m})-\sqrt{m}a\zeta^{2}}\frac{\sqrt{2\pi}ie^{-i\Arg(\phi_{z^{2}}(\zeta,s+2a\zeta/\sqrt{m}))/2}}{\zeta\sqrt{m}\sqrt{\left|\phi_{z^{2}}(\zeta,s+2a\zeta/\sqrt{m})\right|}}\nonumber \\
 & =\frac{e^{ms\zeta+a\sqrt{m}\zeta^{2}-m\zeta^{4}}}{\zeta^{m+1}}\frac{\sqrt{2\pi}ie^{-i\Arg(\phi_{z^{2}}(\zeta,s+2a\zeta/\sqrt{m}))/2}}{\sqrt{m}\sqrt{\left|\phi_{z^{2}}(\zeta,s+2a\zeta/\sqrt{m})\right|}}\label{eq:gsformula}, \qquad \textrm{for} \quad s \in J.
\end{align}
The goal of this section is to show that the change in the argument of $h(s)$ is essentially the same as that of $g(\zeta(s))$. Our first result is the following.
~ 
\begin{lem}
\label{lem:argglobal} Let $C_{1}, C_{2} >0$ be fixed, and $m \gg1$ be such that
\[
I_{1}=\left(C_{1}\ln m/m,2^{5/2}/3^{3/4}-\frac{2a}{\sqrt[4]{12}\sqrt{m}}-C_{2}e^{-m^{1/8}}\right)
\]
is a proper subinterval of $J$. Then 
\[
\Delta\arg_{s\in I}h(s)=\Delta\arg_{s\in I}g(s)+\delta_{1}
\]
for some $\delta_1$ satisfying $|\delta_{1}|\le\pi/2+o(1)$. 
\end{lem}

\begin{proof}
We recall from Section \ref{sec:localasymp} that if $s$ lies in
an interval $\tilde{I} \subset J_1$ satisfying
\[
e^{-m^{1/8}}\ll2^{5/2}/3^{3/4}-s-\frac{2a}{\sqrt[4]{12}\sqrt{m}}=o(1),
\]
then
\begin{align*}
h(s) & \sim\frac{e^{-\sqrt{m}a\zeta^{2}+\phi\left(\zeta,s+\frac{2a\zeta}{\sqrt{m}}\right)}}{\zeta}\int_{\alpha}^{\beta}e^{-my^{2}}z'(y)dy,
\end{align*}
where 
\[
\Re\left(e^{-3\pi i/4}\int_{\alpha}^{\beta}e^{-my^{2}}z'(y)dy\right)\ge0.
\]
Moreover, when 
\[
2^{5/2}/3^{3/4}-s-\frac{2a}{\sqrt[4]{12}\sqrt{m}}=o(1)
\]
and 
\[
1/m^{1/2}=o\left(2^{5/2}/3^{3/4}-s-\frac{2a}{\sqrt[4]{12}\sqrt{m}}\right),
\]
we have $z(y)-\zeta\asymp1/m^{1/4}=o(\zeta-1/\sqrt[4]{12})$ for $\alpha\le y\le\beta$.
Hence, for such values of $s$, the asymptotic expression for $z'(y)$ in equation \eqref{eq:zderivasymp} implies that
\[
\Arg\left(e^{-3\pi i/4}\int_{\alpha}^{\beta}e^{-my^{2}}z'(y)dy\right)=o(1).
\]
Thus, if we write
\[
\Delta_{\tilde{I}}\arg\int_{\alpha}^{\beta}e^{-my^{2}}z'(y)dy=\delta_{1},
\]
then $|\delta_{1}|\le\pi/2+o(1)$, and
\[
\Delta\arg_{\tilde{I}}h(s)=\Delta\arg_{\tilde{I}}\left(\frac{e^{-\sqrt{m}a\zeta^{2}+\phi\left(\zeta,s+\frac{2a\zeta}{\sqrt{m}}\right)}}{\zeta}\right)+\delta_{1}.
\]
The identity 
\[
\phi_{z^{2}}(\zeta,s+2a\zeta/\sqrt{m})=-2^{7/2}3^{3/4}(\zeta-1/\sqrt[4]{12})+\mathcal{O}\left((\zeta-1/\sqrt[4]{12})^{2}\right)
\]
gives 
\[
\Delta\arg_{s\in\tilde{I}}\left(\frac{1}{\phi_{z^{2}}(\zeta,s+2a\zeta/\sqrt{m})}\right)=o(1),
\]
and we conclude that 
\[
\Delta\arg_{\tilde{I}}h(s)=\Delta\arg_{I_{1}}g(s)+\delta_{1}.
\]
Then the result now follows from the fact that 
that $h(s)\sim g(\zeta(s))$ for $s-2^{5/2}/3^{3/4}=\Omega(1)$ and $s\gg\ln m/m$ (see the conclusion of Section \ref{sec:globalasymp})
\end{proof}
Next, we examine the behavior of $g(\zeta(s))$ near the left endpoint of $J_1$.
\begin{lem}
\label{lem:arglocal}For any fixed $C>0$ and $m \gg 1$, 
\[
\Delta\arg_{s\in(0,C\ln m/m)}g(\zeta(s))=o(1).
\]
\end{lem}

\begin{proof}
Using a computer algebra system, we find that for $s\in(0,C\ln m/m)$ the following equality holds:
\[
\zeta(s)=\frac{1-i}{2}+\frac{(i+1)a}{8\sqrt{m}}+\frac{(i-1)a}{64m}+\frac{is}{16}+\mathcal{O}\left(\frac{\ln^{2}m}{m^{2}}\right).
\]
Thus, 
\begin{align*}
\Delta\arg_{s\in(0,C\ln m/m)}e^{-\zeta^{4}} & =-\Delta\Im_{s\in(0,C\ln m/m)}\zeta^{4}\\
 & =-\Delta\Im_{s\in(0,C\ln m/m)}4\left(\frac{1-i}{2}+\frac{(i+1)a}{8\sqrt{m}}+\frac{(i-1)a}{64m}\right)^{3}\frac{is}{16}+\mathcal{O}\left(\frac{\ln^{2}m}{m^{2}}\right)\\
 & =\frac{s}{16}+\mathcal{O}\left(\frac{1}{m^{3/2}}\right),
\end{align*}
and 
\begin{align*}
\Delta\arg_{s\in(0,C\ln m/m)}\frac{1}{\zeta(s)} & =-\Delta\arg_{s\in(0,C\ln m/m)}\zeta(s)\\
 & =-\Delta\arg_{s\in(0,C\ln m/m)}\left(1+\frac{is}{16\left(\frac{1-i}{2}+\frac{(i+1)a}{8\sqrt{m}}+\frac{(i-1)a}{64m}\right)}+\mathcal{O} \left(\frac{\ln^2 m}{m^2} \right)\right)\\
 & =-\Delta\arg_{s\in(0,C\ln m/m)}\left(1+\frac{is}{8(1-i)}\right)+\mathcal{O}\left(\frac{\ln ^2 m}{m^{3/2}}\right)\\
 & =-\frac{s}{16}+\mathcal{O}\left(\frac{\ln^2 m}{m^{3/2}}\right).
\end{align*}
The conclusion is now immediate from the definition of $g(\zeta(s))$. 
\end{proof}
By combining the results of the preceding two Lemmas with Remark
\eqref{rem:anglezderiv}, we find that
\begin{align*}
 & \Delta\arg_{s\in I_{1}}h(s)\\
= & \Delta\arg_{s\in\left(0,2^{5/2}/3^{3/4}-\frac{2a}{\sqrt[4]{12}\sqrt{m}}-C_{2}e^{-m^{1/8}}\right)}g(s)+\delta_{1}.
\end{align*}
Completely analogous arguments show that if 
\[
I_{2}=\left(C_{1}\ln m/m,2^{5/2}/3^{3/4}+\frac{2a}{\sqrt[4]{12}\sqrt{m}}-C_{2}e^{-m^{1/8}}\right),
\]
then 
\begin{align*}
 & \Delta_{s\in iI_{2}}h(s)\\
= & \Delta\arg_{s\in i\left(0,2^{5/2}/3^{3/4}-\frac{2a}{\sqrt[4]{12}\sqrt{m}}-C_{2}e^{-m^{1/8}}\right)}g(s)+\delta_{2},
\end{align*}
where $|\delta_{2}|\le\pi/2+o(1)$. It follows from the definition of $g(s)$
in \eqref{eq:gsformula} that
\[
\left|\Delta\arg_{s\in I_{1}}h(s)\right|+\left|\Delta_{s\in iI_{2}}\arg h(s)\right|
\]
is at least 
\[
|f(s_{1})-f(is_{2})|-\delta,
\]
where
\begin{align*}
0 &<\delta\le\pi+o(1), \\
f(s) & =\Im\left(ms\zeta+a\sqrt{m}\zeta^{2}-\zeta^{4}\right)-(m+1)\Arg\zeta-\Arg(\phi_{z^{2}}(\zeta,s+2a\zeta/\sqrt{m}))/2,\\
s_{1} & =2^{5/2}/3^{3/4}-\frac{2a}{\sqrt[4]{12}\sqrt{m}}-C_{2}e^{-m^{1/8}}, \quad \textrm{and}\\
s_{2} & =2^{5/2}/3^{3/4}+\frac{2a}{\sqrt[4]{12}\sqrt{m}}-C_{2}e^{-m^{1/8}}.
\end{align*}
Recall from Lemma \ref{lem:zetalarges} that
\begin{align*}
\zeta\left(s_{1}\right) & =\frac{1}{\sqrt[4]{12}}+\mathcal{O}\left(e^{-m^{1/8}/2}\right),\\
\zeta(is_{2}) & =-\frac{i}{\sqrt[4]{12}}+\mathcal{O}\left(e^{-m^{1/8}/2}\right).
\end{align*}
Since 
\[
\phi_{z^{2}}(\zeta,s+2a\zeta/\sqrt{m})=-2^{7/2}3^{3/4}(\zeta-1/\sqrt[4]{12})+\mathcal{O}\left((\zeta-1/\sqrt[4]{12})^{2}\right),  \qquad (s\in J_{1})
\]
we conclude that
\[
f(s_{1})=-\frac{\pi}{2}+o(1).
\]
Similarly, the identity 
\[
\phi_{z^{2}}(\zeta,s+2a\zeta/\sqrt{m})=i2^{7/2}3^{3/4}(\zeta+i/\sqrt[4]{12})+\mathcal{O}\left((\zeta+i/\sqrt[4]{12})^{2}\right), \qquad (s\in iJ_{2})
\]
yields 
\[
f(is_{2})=\frac{m+1}{2}\pi-\frac{\pi}{2}+o(1).
\]
Consquently, 
\[
\left|\Delta\arg_{s\in I_{1}}h(s)\right|+\left|\Delta_{s\in iI_{2}}\arg h(s)\right|\ge\frac{m+1}{2}\pi-\delta
\]
for some $0<\delta\le\pi+o(1)$. It follows that the number of nonzero real or purely imaginary
zeros of $H_{m}(s)$ counting multiplicity is at least 
\[
2\left\lfloor \frac{m+1}{2}-1+o(1)\right\rfloor \ge\begin{cases}
m-2 & \text{ if }m\text{ is even}\\
m-3 & \text{ if }m\text{ is odd}
\end{cases}.
\]
Since $H_{m}(s)$ real polynomial of degree $m$ which is even (odd) if $m$ is even (odd) -- for a nice explicit formula\footnote{after the suitable change of variables} see \eqref{eqn:explicitrep},-- it follows that
if $s$ is a (complex) zero of $H_{m}(s)$, then so are $-s$, $\overline{s}$,
and $-\overline{s}$. We conclude from the lower bound
above that all zeros of $H_{m}(s)$ must be either real or purely
imaginary. This completes the proof of Theorem \ref{thm:maintheorem}.


\section{Combinatorial properties}\label{sec:combinatorics}

In this last section we give a couple of combinatorial results concerning Sheffer sequences heretofore discussed from an analytical perspective. Let ${\mathbb{C}}[[z]]$ be the ring of formal power series in variable
$z$ over the complex field $\mathbb{C}$. A polynomial sequence $\left\{ P_{n}(x)\right\} _{n=0}^{\infty}$
is said to be a \textit{Sheffer sequence for $(g,f)$} if there exist
$g,f\in{\mathbb{C}}[[z]]$ where $g(0)\ne0$, $f(0)=0$ and $f'(0)\ne0$
such that 
\begin{eqnarray}
g(z)e^{xf(z)}=\sum_{n=0}^{\infty}P_{n}(x)\frac{z^{n}}{n!},\label{sheffer}
\end{eqnarray}
where $P_{n}(x)$ is a polynomial of degree $n$. From the definition
it follows at once that if $P_{n}(x)=\sum_{k=0}^{n}a_{n,k}x^{k}$
then $\left\{ P_{n}(x)\right\} _{n\ge0}$ can be rewritten as a matrix
product: 
\begin{eqnarray}
(P_{0}(x),P_{1}(x),\ldots)^{T}=AX,\label{matrix}
\end{eqnarray}
where $A=[a_{n,k}]_{n,k\ge0}$ is the coefficient matrix of the Sheffer
sequence $\left\{ P_{n}(x)\right\} _{n=0}^{\infty}$, and $X=(1,x,x^{2},\ldots)^{T}$.
It can be easily shown that $\left\{ P_{n}(x)\right\} _{n=0}^{\infty}$
is a Sheffer sequence for $(g,f)$ if and only if its coefficient
matrix $A=[a_{n,k}]_{n,k\ge0}$ is an exponential Riordan matrix denoted
by $[g,f]$ and defined by 
\begin{eqnarray}
a_{n,k}=\frac{n!}{k!}[z^{n}]gf^{k},\label{coefficient}
\end{eqnarray}
where $[z^{n}]$ is the notation for the coefficient extraction operator.

In this section, we consider polynomials in the sequence $\left\{ P_{n}(x)\right\} _{n=0}^{\infty}$
generated by 
\begin{eqnarray}
\sum_{n=0}^{\infty}P_{n}(x)\frac{z^{n}}{n!}=e^{xz+az^{2}+bz^{4}}.\label{poly}
\end{eqnarray}
By (\ref{sheffer}) and (\ref{matrix}) the generated sequence is a Sheffer sequence with
the coefficient matrix given by the exponential Riordan matrix: 
\begin{eqnarray}
A=\left[e^{az^{2}+bz^{4}},z\right]=[a_{n,k}]_{n,k\ge0}.\label{A}
\end{eqnarray}
Since 
\begin{eqnarray*}
[z^{n}]e^{az+bz^{2}} & = & \sum_{j=0}^{n}[z^{n-j}]e^{az}[z^{j}]e^{bz^{2}}=\sum_{j=0}^{\lfloor n/2\rfloor}[z^{n-2j}]e^{az}[z^{2j}]e^{bz^{2}}\\
 & = & \sum_{j=0}^{\lfloor n/2\rfloor}\frac{a^{n-2j}}{(n-2j)!}\frac{b^{j}}{j!}:=c_{n}(a,b),
\end{eqnarray*}
we obtain 
\begin{eqnarray}
e^{az^{2}+bz^{4}} & = & \sum_{n=0}^{\infty}c_{n}(a,b)z^{2n}=\sum_{n=0}^{\infty}(2n)!c_{n}(a,b)\frac{z^{2n}}{(2n)!}\label{sum}\\
 & = & 1+2a\frac{z^{2}}{2!}+12(a^{2}+2b)\frac{z^{4}}{4!}+120a(a^{2}+6b)\frac{z^{6}}{6!}\cdots.\nonumber 
\end{eqnarray}
It now follows from (\ref{coefficient}), (\ref{A}) and (\ref{sum})
that if $n-k$ is even, then 
\begin{eqnarray}
a_{n,k}=\frac{n!}{k!}[z^{n-k}]e^{az^{2}+bz^{4}}=\frac{n!}{k!}c_{\frac{n-k}{2}}(a,b),\label{ank}
\end{eqnarray}
and $a_{n,k}=0$ otherwise. Moreover, since $P_{n}(x)=\sum_{k=0}^{n}a_{n,k}x^{k}$,
it immediately follows from (\ref{ank}) that for $m=0,1,\ldots$
\begin{eqnarray} \label{eqn:explicitrep}
 &  & P_{2m}(x)=\sum_{k=0}^{m}\frac{(2m)!}{(2k)!}c_{m-k}(a,b)x^{2k},\\
 &  & P_{2m+1}(x)=\sum_{k=0}^{m}\frac{(2m+1)!}{(2k+1)!}c_{m-k}(a,b)x^{2k+1}.
\end{eqnarray}

A few rows of the matrix $A$ are shown below:
\begin{eqnarray}\label{coeff}
A=\left[\begin{array}{ccccccc}
1\\
0 & 1 &  &  & O\\
2a & 0 & 1\\
0 & 6a & 0 & 1\\
12(a^{2}+2b) & 0 & 12a & 0 & 1\\
0 & 60(a^{2}+2b) & 0 & 20a & 0 & 1\\
\vdots &  & \cdots &  &  &  & \ddots
\end{array}\right].
\end{eqnarray}

\begin{thm}\label{thm16}
The Sheffer sequence $\left\{ P_{n}(x)\right\} _{n=0}^{\infty}$ generated by \eqref{poly} satisfies
the following recurrence relation for $n\ge4$:
\begin{eqnarray}
{\label{recurrence}}P_{n}(x)=xP_{n-1}(x)+2a{n-1 \choose 1}P_{n-2}(x)+24b{n-1 \choose 3}P_{n-4}(x),
\end{eqnarray}
where $P_{0}(x)=1$, $P_{1}(x)=x$, $P_{2}(x)=2a+x^{2}$, and $P_{3}(x)=6ax+x^{3}$.
\end{thm}

\begin{proof}
For $k\ge0$ let ${\rm P}_{k}(x)=\left(P_{k}(x),P_{k+1}(x),\ldots\right)^{T}$.
We first note that (\ref{recurrence}) is equivalent to the matrix
equation:
\begin{eqnarray}
{\rm P}_{1}(x)=(B+xI){\rm P}_{0}(x),\label{me}
\end{eqnarray}
where
\begin{eqnarray*}
B=\left[\begin{array}{ccccccc}
0\\
2a & 0 &  &  & O\\
0 & 4a & 0\\
24b & 0 & 6a & 0\\
0 & \ddots & 0 & 8a & 0\\
0 & 0 & 24b{n-1 \choose 3} & 0 & 2a{n-1 \choose 1} & 0\\
\vdots & \cdots & \ddots & \ddots & \ddots & \ddots & \ddots
\end{array}\right]
\end{eqnarray*}
and $I$ is the infinite identity matrix. Moreover, ${\rm P}_{0}(x)=AX$
from (\ref{matrix}). Since $A^{-1}=\left[e^{-az^{2}-bz^{4}},z\right]$
it can be easily shown that 
\begin{eqnarray}
A^{-1}UA=B+U,\label{Stieltzes}
\end{eqnarray}
where $U$ is the upper shift matrix with ones only on the superdiagonal
and zeros elsewhere. Thus we have
\begin{eqnarray}
{\rm P}_{1}(x)=U{\rm P}_{0}(x)=UAX=(AB+AU)X=(AB)X+(AU)X.\label{p1x}
\end{eqnarray}
Since $B$ can be expressed as the exponential Riordan matrix of the
Appell form given by $\left[2az+24bz^{3}/3!,z\right]$, we immediately
obtain $AB=BA$. Thus
\begin{eqnarray*}
(AB)X=(BA)X=B(AX)=B{\rm P}_{0}(x).
\end{eqnarray*}
Clearly, $(AU)X=A(UX)=x{\rm P}_{0}(x)$. Hence it follows from (\ref{p1x})
that
\begin{eqnarray*}
{\rm P}_{1}(x)=B{\rm P}_{0}(x)+x{\rm P}_{0}(x)=(B+xI){\rm P}_{0}(x),
\end{eqnarray*}
which proves (\ref{me}).
\end{proof}

We now ask the reader to consider a {\it marked generating tree}, in order to give combinatorial meaning to the coefficients of a Sheffer sequence $\left\{ P_{n}(x)\right\} _{n=0}^{\infty}$. A marked generating tree (\cite{{BFS},MSV}) is a labeled tree with a root, where non-marked or marked labels of nodes at each level $n\ge0$ are determined by a specified rule called the {\it succession rule}. By convention, the root is a node with label 0 at level $0$. The nodes with a non-marked label $k$ and a marked label $\bar k$ will be denoted by $(k)$ and $(\bar k)$ respectively.
The works \cite{DFR,MSV}, studied a method to represent the succession rule by an integer matrix $P = [p_{k,i}]_{k,i\ge0}$, called  the {\it production matrix} of the marked generating tree the following way: for integers $i$ with $0\le i\le k+1$, if $p_{k,i}\ge0$ then $p_{k,i}$ is the number
of non-marked nodes $(i)$ produced by node $(k)$, denoted by $(i)^{p_{k,i}}$ ; and if $p_{k,i}<0$ then $|p_{k,i}|$ is the number
of marked nodes $(\bar i)$ produced by node $(k)$, denoted by $(\bar i)^{|p_{k,i}|}$. Then the succession rule for each node $(k)$ with $k\ge0$ turns out as follows:
\begin{align}
\left\{
\begin{tabular}{ll}
root : & $(0)$ \\
rule : & $(k)\rightarrow (0)^{p_{k,0}}(1)^{p_{k,1}}\cdots(k+1)^{p_{k,k+1}}$
\end{tabular}%
\right.  \label{eq:production}
\end{align}
where if $p_{k,i}<0$ then $(i)^{p_{k,i}}$ is converted to $(\bar i)^{|p_{k,i}|}$, and $({\bar{\bar i}})=(i)$.
\vskip.5pc

The following lemma follows immediately from \cite{CFKT,{DFR}, MSV}. The basic idea is that marked labels annihilate the non-marked labels with the same number at the same level.
\begin{lem}\label{lem17} Let $A=[a_{n,k}]$ be an exponential Riordan matrix of integer entries with all diagonal elements equal to 1. Then there exists a marked generating tree associated to $A$ such that
\begin{itemize}
\item[{\rm(a)}] the succession rule (\ref{eq:production}) is obtained from the production matrix $P=[p_{k,i}]$ given by $P=A^{-1}UA$ for the upper shift matrix $U$.
\item[{\rm(b)}] $a_{n,k}$ can be interpreted combinatorially by the {\it difference} between the numbers of nodes $(k)$ and $(\bar k)$ at level $n$.
\end{itemize}
\end{lem}

By the recurrence relation in Theorem \ref{thm16} and Lemma \ref{lem17}, we arrive at the following theorem.

\begin{thm} Let $\left\{ P_{n}(x)\right\} _{n=0}^{\infty}$ be the Sheffer sequence with generating function $e^{xz+az^{2}+bz^{4}}$
 where $a,b$ are integers with $a>0$ and $b<0$. For $k=0,1,\ldots,n$, let
$\mu_{n}(k)$ and $\mu_{n}(\bar k)$ respectively denote the numbers of non-marked nodes $(k)$ and marked nodes $(\bar k)$ at level
$n$ in the marked generating tree specification:
\begin{align}\label{eq:gts}
\left\{ \begin{tabular}{ll}
 root :  &  (0)\\
 rule :   &  \ensuremath{{(k)}\rightarrow{(\overline{k-3})}^{24|b|{\binom{k}{3}}}({k-1}\ensuremath{)^{2ak}\;}(k+1)^1},\\
 &  \ensuremath{(\bar k)\rightarrow{(k-3)}^{24|b|{k \choose 3}}\;{(\overline{k-1})}^{2ak}\;{(\overline{k+1})^1}}
\end{tabular}\right.
\end{align}
where $(i)^0$ is the empty node. Then
\begin{eqnarray}
[x^{k}]P_{n}(x)=\mu_{n}(k)-\mu_{n}(\bar k).\label{e:interpretation-1}
\end{eqnarray}

\end{thm}
\begin{proof} Let $A=[a_{n,k}]$ be the coefficient matrix with $a_{n,k}=[x^{k}]P_{n}(x)$. Since $a,b$ are integers, from (\ref{A}) and (\ref{coeff}) we see that $A$ is the exponential Riordan matrix $\left[e^{az^{2}+bz^{4}},z\right]$ whose entries are integers with all diagonal elements equal to 1. By Lemma \ref{lem17} there exists a marked generating tree associated to $A$ such that the production matrix is given by $P=A^{-1}UA$. Thus it follows from (\ref{Stieltzes}) that $B+U$ is the production matrix of the generating tree. Let $B+U=[p_{n,k}]$. Then we obtain
\begin{eqnarray}\label{succe}
p_{k,k+1}=1,\;p_{k,k-1}=2ak,\;p_{k,k-3}=24b{k\choose 3},\;{\rm otherwise}\;p_{n.k}=0.
\end{eqnarray}
Since $b<0$ the succession rule (\ref{eq:gts}) follows from (\ref{eq:production}) and (\ref{succe}). Moreover, by (b) of Lemma \ref{lem17} we have (\ref{e:interpretation-1}), as required.
\end{proof}

\begin{example} For instance, from rule (\ref{eq:gts}) we obtain the succession rule for $k=0,1,2,3$:
\begin{eqnarray*}
(0)\rightarrow (1)^1,\;(1)\rightarrow (0)^{2a}(2)^1,\;(2)\rightarrow (1)^{4a}(3)^1,\;(3)\rightarrow (\bar 0)^{24|b|}(2)^{6a}(4)^1.
\end{eqnarray*}
Thus the marked generating tree has the following nodes at each level $n=0,1,2,3,4$, with the associated polynomials $P_n(x)$ as indicated:
\begin{eqnarray*}
&&n=0:(0);\quad\quad\quad\quad\quad\quad\quad\quad\quad P_0(x)=1\\
&&n=1:(1);\quad\quad\quad\quad\quad \quad\quad\quad\quad P_1(x)=x\\
&&n=2:(0)^{2a}(2);\quad\quad\quad\quad\quad\quad\quad P_2(x)=2a+x^2\\
&&n=3:(1)^{6a}(3);\quad\quad\quad\quad\quad\quad\quad P_3(x)=6ax+x^3\\
&&n=4:(0)^{12a^2}(\bar 0)^{24|b|}(2)^{12a}(4);\quad P_4(x)=(12a^2+24b)+12ax^2+x^4.
\end{eqnarray*}
\end{example}


\begin{thebibliography}{99}

\bibitem{BFS} F. Bergeron, P. Flajolet, and B. Salvy, Varieties of
increasing trees, In CAAP'92 (Rennes, 1992), volume 581 of Lecture
Notes in Comput. Sci., pages 24-48. Springer, Berlin, 1992.

\bibitem{CFKT} G.-S. Cheon, T. Forg\'acs, H. Kim and K. Tran, On combinatorial properties and the zero distribution of certain Sheffer sequences, {\it J. Math. Anal. Appl.} {\bf 514} (2022), DOI:10.1016/j.jmaa.2022.126273



\bibitem{DFR} E. Deutsch, L. Ferrari, and S. Rinaldi, Production
matrices and Riordan arrays, Ann. Comb. 13 (2009), 65--85.

\bibitem{tk1} T. Forg\'acs and K. Tran, Polynomials with rational generating functions and real zeros. {\it J. Math. Anal. Appl.} {\bf 443}(2) (2016), pp. 631-651. DOI 10.1016/j.jmaa.2016.05.041 

\bibitem{tk2} T. Forg\'acs and K. Tran, Zeros of polynomials generated by a rational function with a hyperbolic-type denominator, {\it Constr. Approx.} {\bf 46} (2017), pp. 617-643.  DOI 10.1007/s00365-017-9378-2     

\bibitem{tk3} T. Forg\'acs, and K. Tran, Hyperbolic polynomials and linear-type generating functions, {\it J. Math. Anal. Appl.} {\bf 488} (2) (2020) DOI: 10.1016/j.jmaa.2020.124085    


\bibitem{MSV} D. Merlini, R. Sprugnoli, M. C. Verri, An Algebra for Proper Generating Trees, Mathematics and
Computer Science: Algorithms, Trees, Combinatorics and Probabilities, (2000) 127-152.
                                                                                                                                                                                                                                              








theory viii: finite operator calculus. J. Math. Anal. Appl. 42 (1973), 684-760.


\bibitem{temme} N. M. Temme., Asymptotic Methods for Integrals. Series in Analysis, vol. 6. ISBN: 978-981-4612-15-9, World Scientific, 2014.

\bibitem{ps} \textit{G.~P\'olya} and \textit{J.~Schur}, \"Uber zwei Arten von Faktorenfolgen in der Theorie der algebraischen Gleichungen, J. Reine Angew. Math., {\bf 144} (1914), 89-113. 


\end{thebibliography}
\end{document}